\documentclass
[
    a4paper,
    DIV=10,
    abstracton,
    10pt
]
{scrartcl}

\usepackage{
    a4wide,
    amsmath,
    amssymb,
    amsthm,
    bm,
    graphicx,
    caption,
    subcaption,
    enumitem,
    xcolor,
    mathtools,
    authblk,
    dsfont,
    listings,
    float,
    algorithmic,
    csvsimple,
    stackengine,
    makecell,
    booktabs,
    multirow,
    bbm,
    mathrsfs,
    xfrac,
    mathrsfs
}

\usepackage[normalem]{ulem}

\usepackage[export]{adjustbox}

\usepackage{tabularx}
\usepackage{array}

\usepackage[section]{algorithm}
\usepackage[percent]{overpic}

\usepackage[utf8]{inputenc}
\usepackage[pdffitwindow=false,
            plainpages=false,
            pdfpagelabels=true,
            pdfpagemode=UseOutlines,
            pdfpagelayout=SinglePage,
            bookmarks=false,
            colorlinks=true,
            hyperfootnotes=false,
            linkcolor=blue,
            urlcolor=blue!30!black,
            citecolor=green!50!black]{hyperref}

\newcommand{\R}{\mathbb{R}}                                     
\newcommand{\X}{\mathbb{X}}                                     
\newcommand{\Y}{\mathbb{Y}}                                     
\newcommand{\Z}{\mathbb{Z}}                                     

\newcommand{\M}{\mathbb{M}}

\newcommand{\Q}{{\mathcal{Q}}}
\providecommand{\abs}[1]{\left\lvert #1 \right\rvert}           
\providecommand{\norm}[1]{\left\lVert #1 \right\rVert}          
\newcommand{\var}{\operatorname{Var}}

\newcommand{\LL}{\mathscr{L}_\text{L}}
\newcommand{\LD}{\mathscr{L}_\text{D}}
\newcommand{\LRL}{\widetilde{\mathscr{L}}_\text{L}}
\newcommand{\LRD}{\widetilde{\mathscr{L}}_\text{D}}

\newcommand{\ts}{\hspace*{0.1em}} 

\DeclareMathOperator*{\argmin}{arg\,min}

\mathtoolsset{centercolon} 
\newtheorem{theorem}{Theorem}[section]
\newtheorem{corollary}[theorem]{Corollary}
\newtheorem{lemma}[theorem]{Lemma}
\newtheorem{proposition}[theorem]{Proposition}
\newtheorem{definition}[theorem]{Definition}
\theoremstyle{definition}
\newtheorem{example}[theorem]{Example}
\newtheorem{remark}[theorem]{Remark}

\newtheorem*{remark*}{Remark}

\definecolor{boxback}{gray}{0.95}

\newcommand{\markchange}[1]{#1}

\graphicspath{{./}}

\title{Optimal Reaction Coordinates: Variational Characterization and Sparse Computation}

\author[1]{Andreas Bittracher}
\author[1]{Mattes Mollenhauer}
\author[1]{Péter Koltai}
\author[1,2]{Christof Schütte}

\affil[1]{Department of Mathematics and Computer Science, Freie Universit\"at Berlin, Germany}
\affil[2]{Zuse Institute Berlin, Germany}
\date{}

\begin{document}

\maketitle


\begin{abstract}
Reaction Coordinates (RCs) are indicators of hidden, low-dimensional mechanisms that govern the long-term behavior of high-dimensional stochastic processes. We present a novel and general variational characterization of optimal RCs and provide conditions for their existence. Optimal RCs are minimizers of a certain loss function and reduced models based on them guarantee very good approximation of the long-term dynamics of the original high-dimensional process. 
We show that, for slow-fast systems, metastable systems, and other systems with known good RCs, the novel theory reproduces previous insight. 
Remarkably, the numerical effort required to evaluate the loss function scales only with the complexity of the underlying, low-dimensional mechanism, and not with that of the full system.
\markchange{The theory provided lays the foundation for an efficient and data-sparse computation of RCs via modern machine learning techniques.}
\end{abstract}


\section{Introduction}

Complex high-dimensional dynamical processes, as they occur in molecular dynamics, statistical mechanics or finance, are  often modelled by time- and space-continuous Markov processes on the microscopic level. The reason researchers are interested in such processes is however often not so much the microscopic dynamics itself, but rather the phenomenon that, over long time scales, the system often exhibits much more regularity and much less complexity than the sheer number of degrees of freedom would actually allow for.
The desire to understand the emergence of this macro-scale behaviour, and to simulate it efficiently, motivates the derivation of reduced, in general non-Markovian models. By projecting the full dynamics onto some selected observables, using techniques such as the Mori-Zwanzig formalism~\cite{zwanzig_nonequilibrium_2001}, the reduced model describes the long-term dynamical behaviour of these observables, while discarding the short-term microscopic detail. For most such observables the resulting reduced dynamics is too complex (long-term memory effects, complicated noise process) to be useful. However, particular observables allow for a rather direct representation of the reduced dynamics. In the context of molecular dynamics, such particular observables are known as \emph{reaction coordinates} (RCs) or \emph{collective variables}; we will adopt the former designation herein.

This article illuminates the first, and arguably most crucial step in the model reduction pipeline, which is the selection of the ``correct'' or "optimal" RCs.
We present a novel, comprehensive theory of optimal RCs based on a variational principle that is accessible by modern machine learning methods. Its key features are described below. Its differentiation from previous approaches can be found in Section~\ref{sec:related_work}. The novel theory presented herein has some crucial advantages: 

\paragraph{Variational characterization with dynamical meaning.}

Our theory is the first to explicitly describe a dynamically interpretable variational principle for optimal RCs that goes beyond alternative approaches based on spectral decomposition or autocorrelation~\cite{perez-hernandez_identification_2013,wehmeyer_time-lagged_2018} by avoiding any linearization steps. We derive a loss function that measures how well a given RC and the associated reduced dynamics  preserve the long-term behaviour of the full system. To quantify the discrepancy in the long-term behaviour, a special metric based on two specific, dynamically interpretable and intuitive properties is used. Optimal RCs, i.e., the global minimizers of the loss function, then inherit this interpretability naturally.

\paragraph{Consistency with established theory.}

The new theory covers system classes for which good RCs are predefined by established theory or are intuitively clear. Applied to slow-fast dynamical systems~\cite{pavliotis_multiscale_2008}, or systems with a timescale gap separating slowly and quickly equilibrating sub-processes~\cite{schutte_metastability_2014}, our variational principle correctly characterizes the respective slow components as optimal RCs. 
Moreover, our theory extends the transition manifold framework~\cite{bittracher_transition_2017}, which has been used to successfully identify RCs in several high-dimensional biomolecular systems~\cite{bittracher_data-driven_2018, bittracher_exploring_2021}. By utilizing transition manifold theory, it can be shown how to control the approximation error between the long-term behavior of the original system and the reduced model for optimal RCs.

\paragraph{Sparse sampling of low-rank dynamics.}

Key parts of our variational problem are insensitive to the curse of dimensionality. By interpreting the existence of a latent low-dimensional reduced model as a low-rank like property, and using arguments similar to low-rank matrix approximation~\cite{halko_finding_2011} and compressive sampling~\cite{foucart_mathematical_2013} techniques, we show that approximating our loss function requires only short simulations started from sparsely sampled starting points. More precisely, the number of required starting points does not scale with the dimension of the full system, but only with the dimension and complexity of the reduced system.

\paragraph{Leverage point for modern machine learning.}

The variational problem can be used directly as a loss function in deep learning methods. We expect that a numerical scheme based on a combination of massively parallel dynamical sampling, automatic differentiation and stochastic gradient descent minimization of our loss function will be effective in identifying interpretable RCs in large-scale real-world systems, such as biomolecular complexes.
\\ 

This article is structured as follows: Section~\ref{sec:related_work} compares our approach to related work. Section~\ref{sec:characterization} introduces our characterization of good and optimal RCs, and confirms the compatibility of that characterization with established concepts. Section~\ref{sec:variational} derives the variational principle in form of a loss function for optimal RCs. Section~\ref{sec:numerical_aspects} shows how this loss function can be efficiently approximated via a Monte Carlo method with sparse samples. It also discusses the applicability of modern machine learning methods. Finally, section~\ref{sec:examples} demonstrates the variational principle and the sparse loss function approximation by two synthetic examples. Section~\ref{sec:conclusions} contains the conclusions, and an outlook on future work.

\section{Related work}
\label{sec:related_work}

\paragraph{Characterization of dynamically meaningful RCs.}
Existing definitions of ``good'' RCs were mostly motivated by applications in computational physics and chemistry. However, these field to this date still predominantly use \emph{heuristic} RCs in their model reduction efforts. This practice of manually selecting and combining coordinates from a pool of candidate physical observables such as bond angles or residue distances comes with obvious shortcomings with respect to dynamical meaningfulness, optimality and scalability. Mostly, one is mainly interested in reproducing statistical quantities like the free energy correctly without  any strict characterization of the dynamical meaning. 
We now discuss the most popular of the few approaches to the identification of dynamically meaningful RCs.

The \emph{committor function} is a one-dimensional reaction coordiante, which for some reactant and some product state indicates the probability to hit the latter before hitting the former~\cite{e_towards_2006, maragliano_string_2006}. However, the committor can only be expected to describe the system's rate-limiting processes well for sensible choices of the reactant and product (which obviously requires a priori macro-scale knowledge of the system).

The dominant eigenfunctions of the system's transfer operator (or equivalently its Fokker Planck operator) are mathematically related to committor functions~\cite{sarich_markov_2014}. These eigenfunctions linearly decompose the system into independent sub-processes, which equilibrate with a rate determined by the associated eigenvalue~\cite{schutte_metastability_2014}. Hence, the dominant eigenfunctions have been used as RCs~\cite{mcgibbon_identification_2017}.
While this correspond well with our definition of a good RC, it has been demonstrated that the dominant eigenfunctions themselves can be reduced further, if the associated sub-processes are in some way ``nonlinearly dependent'' on each other~\cite{bittracher_transition_2017}. Our second example system (Section~\ref{sec:example_metastable_circular}) demonstrates this situation. A RC composed of dominant eigenfunctions is therefore in general not optimal in terms of its dimensionality.

The TICA (time-lagged independent component analysis) method constructs RCs as those linear combinations of the original degrees of freedom with the highest autocorrelation~\cite{molgedey_separation_1994, perez-hernandez_identification_2013}. Assuming the system is reversible, the TICA coordinates are the aforementioned eigenfunctions of the transfer operator projected onto the linear basis functions $\psi_i(x)=x_i$~\cite{klus_eigendecompositions_2020}. They therefore suffer from the same non-optimality as the eigenfunction RCs, and in addition from non-optimality due to the overhead of linear approximation.

The most frequently analyzed special case of timescale-separated systems are \emph{slow-fast systems} (see~\cite{pavliotis_multiscale_2008} for a text book introduction). They are characterized by the existence of a coordinate transformation such that the new coordinates can be subdivided into one quickly and one slowly moving part, and the two parts are approximately decoupled. The slow coordinates, which form a parametrization of the system's slow manifold~\cite{wang_slow_2013}, can be considered, and can be used as a good RC of the system~\cite{singer_detecting_2009,froyland_computational_2014}.
We will see later that our characterization encompasses that of slow variables.

The RCs conceptually most alike to the definition presented in this article are the \emph{transition manifold RCs}, proposed in~\cite{bittracher_transition_2017} and further refined in~\cite{bittracher_weak_2020}. Transition manifold theory characterize good RCs as a parametrization of a low-dimensional manifold in a certain function space. This \emph{transition manifold} emerges from the system's transition densities with progressing equilibration. The characterization presented in the present article is a significant generalization of~\cite{bittracher_weak_2020}.

\paragraph{Computational strategies.}

Most of the aforementioned approaches to dynamically meaningful RCs come with a proposed numerical scheme for their computation. The committor function, which satisfies a backward Kolmogorov equation~\cite{e_towards_2006}, can be computed using numerical PDE solvers (although this was never proposed as a practical scheme and a vastly more efficient scheme was proposed soon-after~\cite{maragliano_string_2006}). RCs based on transfer operator eigenfunctions (including TICA) can be computed by an eigendecomposition of a suitable discretization of that operator~\cite{dellnitz_approximation_1999, schutte_direct_1999, deuflhard_robust_2005, perez-hernandez_identification_2013}.
Approaches that characterize RCs as parametrization of some manifold use unsupervised manifold learning methods such as diffusion maps to learn the variables in an equation-free manner~\cite{singer_detecting_2009, froyland_computational_2014, bittracher_transition_2017, bittracher_dimensionality_2020}.

Over time, deep-rooted relationships between the different RCs, as well as extensions and generalizations were discovered (see~\cite{klus_data-driven_2018} for a partial overview), which led to alternative and more efficient schemes for their computation.
While a comprehensive listing would go beyond the scope of this article, we want to point out an emerging trend in these efforts: the formulation of a variational principle for the respective RC. There now exist variational approaches for the committor function~\cite{khoo_solving_2018}, the TICA coordinate~\cite{wehmeyer_time-lagged_2018}, the transfer operator eigenfunctions~\cite{mardt_vampnets_2018}, and related dominant subspaces~\cite{RabbenRayWeber2020}.
Alternative approaches try to maximize timescale separation between slow and fast processes via maximum caliber-based frameworks \cite{Tiwary2020}.
The driving force behind this trend is of course the desire to profit from the impressive performance that modern deep learning and neural network-based methods have demonstrated with regard to their generalization power, robustness to overfitting and seeming immunity to the curse of dimensionality~\cite{berner_modern_2021}.

Notably missing from the above list is however a variational principle for manifold-based RCs. As mentioned before, the characterization presented in this article generalizes the transition manifold, which in turn generalizes the slow manifold, so it can be seen as a completion in that regard.
The variational approach then offers the additional advantage of yielding a closed form of the RC (in some finite-dimensional ansatz space), unlike the aforementioned geometric manifold learning algorithms, which output only discrete point-evaluations of the RC.

\paragraph{Sparsity.}

Our strategy to approximate the loss function from sparse samples of the dynamics shows parallels to other computational techniques that implicitly exploit some form of hidden regularity of the problem.

In a recent publication~\cite{bittracher_probabilistic_2021}, a discrete version of the sparse sampling strategy to time- and space-\emph{discrete} Markov chains was proposed.
There it was assumed that the long-term behavior of a large (i.e., many-state) Markov chain is essentially determined by transitions between certain \emph{aggregates} of these states. It was shown that the aggregates and the transition probabilities between them could be discovered from a vastly undersampled version of the transition matrix of the original chain (obtainable through simulations starting from random states). Some of the key ideas of \cite{bittracher_probabilistic_2021} will be exploited herein by putting them into the continuous setting.

The fundamental idea behind both the discrete and continuous sparse sampling strategies is heavily inspired by the field of compressive sensing, see~\cite{foucart_mathematical_2013} for an introduction. The impressive feat of compressive sensing is its ability to re-construct a signal (a high-dimensional vector) from far less samples than the Nyquist–Shannon sampling theorem would actually demand by means of solving vastly underdetermined linear equations. The necessary assumption is that the system is \emph{sparse} in some basis (for example the frequency domain), though the location or precise number of the sparse entries does not need to be known. In a way, the dominance of a Markov process by a single low-dimensional mechanism can be interpreted as a sort of ``non-linear dynamical sparsity''.

Finally we would like to point out that the apparent similarity of our sampling strategy to randomized matrix low-rank approximation techniques~\cite{halko_finding_2011} like the Nyström method~\cite{drineas_nystrom_2005} or randomized feature approximation~\cite{rahimi_random_nodate} is rather superficial. \markchange{While for these techniques a (nearly) low-rank structure of the target matrix is necessary to achieve low approximation error, they do not interpret this low rank as an underlying structure ``generating'' the matrix}. Indeed, in~\cite{udell_why_2019} it has been argued that an (approximate) low rank is a generic property of large data matrices, and that the attribution of that rank to some ``physical reason'' is in general not possible.
Consequently, this rank and with it the sampling requirement still scales with the size of the matrix (typically logarithmically~\cite{udell_why_2019}). In contrast, our theory asssumes and exploits the existence of an underlying mechanism of fixed low dimension that governs the micro-scale model, which effectively decouples the sampling requirement from the size of the full model.

\section{Characterization of good RCs}
\label{sec:characterization}

This section contains the central definition of good RCs and corresponding conditions on the dynamics for their existence. First, however, we introduce the notation and the dynamical setting.

\subsection{Definition of the dynamics and fundamental assumptions}

Let $\X\subset\R^n$ be a Lebesgue-measurable set (the \emph{state space}) and $(X_t)_{t\in\R^+}$, or short $(X_t)$, be a time- and space-continuous Markov process on $\X$. Let $P^t: \X \times \mathscr{B} \rightarrow [0,1]$ denote the \emph{transition probability function} of $(X_t)$, where $\mathscr{B}$ is the Borel $\sigma$-algebra on $\X$, i.e.,
$$
P^t[x,B] = \operatorname{Prob}[X_{t_0+t} \in B ~|~X_{t_0} = x] \quad \text{for all } t_0\geq 0.
$$
For any $t>0$ and $x\in\X$, $P^t(x,\cdot)$ is a probability measure on $\mathscr{B}$, and $P^t(\cdot,B)$ is a $\mathscr{B}$-measurable function for any $B\in\mathscr{B}$~\cite{rosenblatt1967}. Moreover, let the process be ergodic, such that a unique stationary measure $\mu:\mathscr{B}\rightarrow \R^+_0$ exists. We require $\mu$ to be absolutely continuous with respect to the Lebesgue measure, i.e., there exists a density $\pi:\X\rightarrow \R^+$ such that
$$
\mu(B) = \int_B \pi(x)\ts dx.
$$
\markchange{Moreover, we require that $\pi$ is continuous and strictly positive.}

We also require the $P^t(x,\cdot)$ to be absolutely continuous with respect to the Lebesgue measure. Thus, we may assume that there exists a family of functions $p^t:\X\times \X \rightarrow \R^+$ such that
\begin{equation}
\label{eq:absolute_continuity}
P^t(x,B) = \int_B p^t(x,y)\ts \mathrm{d}y \quad \text{for all } \tau>0,~x\in\X,~B\in\mathscr{B}.
\end{equation}
Many classes of Markov processes are absolutely continuous, including Itô diffusions with smooth coefficients~\cite{krengel_ergodic_1985}.
Also, we assume that the system is reversible with respect to $\pi$, i.e., the \emph{detailed balance equation} holds:
\begin{equation}
\label{eq:detailed_balance}
	p^t(x,y) \pi(x) = p^t(y,x) \pi(y)\quad \text{for all } x,y\in \X,~t\in\mathbb{R}^+.
\end{equation}

The existence of good RCs will be determined by specific properties of the function $p^t$, hence we now examine it more closely. As a function of the second argument, $p^t(x,\cdot)\in L^1$ is the time-$t$ \emph{transition density function} of $(X_t)$, i.e.
$$
p^t(x,\cdot) = \operatorname{Law}\big(X_{t_0+t}~\big|~X_{t_0} = x\big) \quad \text{for all } t_0\geq 0.
$$

On the other hand, $p^t(\cdot,y)$ as a function of the first argument is harder to interpret, and discussed less in the literature of stochastic processes. Let $L^1_\mu(\X)$ be the space equipped with the norm
$$
\|f\|_{L^1_\mu} := \int_\X f(x)\ts d\mu(x).
$$
\markchange{We then have $p^t(\cdot,y)\in L^1_\mu$, since, by reversibility of $X_t$,
\begin{align*}
\int_\X p^t(x,y) \ts d\mu(x) &= \int_\X p^t(x,y) \pi(x)\ts dx \\
&= \int_\X p^t(y,x) \pi(y)\ts dx \leq \|\pi\|_\infty\ts \|p^t(y,\cdot)\|_{L^1} < \infty.
\end{align*}
To distinguish it from the transition density, we call $p^t(\cdot,y)$ the time-$t$ \emph{transition observable} of $y$. }

As a function of \emph{two} arguments, we call $p^t:\X\times\X\rightarrow\R$ the time-$t$ \emph{transition kernel} of $(X_t)$. It can be interpreted as an
element of the space of functions
$L^1_{\mu \times \lambda}( \X^2)$, where
$\mu \times \lambda$ is the product measure on the space
$\X^2$ given by the invariant measure $\mu$
and the Lebesgue measure $\lambda$ on $\X$.
For simplicity, we use the shorthand notation
$$
\mathbb{K} := L^1_{\mu \times \lambda}( \X^2).
$$
Note that by Fubini--Tonelli, we have
\begin{align}
\label{eq:K_norm}
\| p(\ast,\cdot) \|_\mathbb{K} &:= \big\| \left\| p(\ast ,\cdot) \right\|_{L^1(\X)} \big\|_{L^1_\mu(\X)}
\end{align}
as the norm on $\mathbb{K}$, where in~\eqref{eq:K_norm} the inner norm applies to the argument ``$\cdot$'', and the outer norm applies to the argument ``$\ast$'' (this will be a convention from now on).

\subsection{Lumpability and decomposability}
\label{sec:lumpability_decomposability}

We will now introduce two seemingly different conditions for a system/RC pair. Each condition may individually be taken as a definition for what a good RC is. It will turn out, however, that the two conditions are equivalent to each other for reversible systems, so a good RC with respect to one condition is a good RC with respect to the other.

 Let~$r < n$, and $\Z \subset \R^r$ be a domain. We call any function $\xi\in C(\X,\Z)$ an \emph{$r$-dimensional RC}. Later, in particular in Section~\ref{sec:variational}, we will require $\xi$ to be ``smooth'':
\begin{definition}[Smooth RCs]
\label{def:smooth_rc}
	Let~$r < n$, and $\Z \subset \R^r$ be a domain. If $\xi:\X\rightarrow \Z$ fulfils the two assumptions
\begin{enumerate}
	\item $\xi$ is continuously differentiable,
	\item for all $z\in\Z$, the level sets
	$$
	\Sigma_\xi(z) := \left\{ x\in\X~|~\xi(x) = z\right\}
	$$
	are $n-r$-dimensional topological submanifolds of $\X$,
\end{enumerate}
we call $\xi$ a \emph{smooth $r$-dimensional RC}. We denote the function space of smooth RCs from $\X$ to $\Z$ by $S(\X,\Z)$.

\end{definition}
In case where $\xi\in S(\X,\Z)$, the \emph{marginal stationary measure} $\mu_z$ on $\Sigma_\xi(z)$ can be defined by
\begin{align}
\label{eq:def_muz}
	\int_{\Sigma_\xi(z)}	f(x)\ts d\mu_z(x)
	&=\int_{\Sigma_\xi(z)}	f(x) \pi(x) \det\big(\nabla\xi(x)^\intercal \nabla\xi(x)\big)^{-1/2} \ts dH_{n-r}(x),
\end{align}
where $H_{d}$ denotes the $d$-dimensional Hausdorff measure. By $|\Z|$, we denote the Lebesgue-measure of $\Z$.

We emphasize, however, that we do not require smoothness for the following characterizations of good RCs.

\subsubsection*{Lumpability}
The first condition on the process $(X_t)$ is as follows:

\begin{definition}[Lumpability]

If there exists a domain $\Z\subset \R^r$, a function $\xi\in C(\X,\Z)$, a family of time-parametrized functions $p_L^t:\Z\times \X\rightarrow \R^+$ and a lag time $\tau>0$ such that
	\begin{equation}
	\label{eq:lumpability}
	\tag{\text{L}}
	\frac{1}{|\Z|} \big\| p^t(\ast,\cdot) - p^t_L(\xi(\ast),\cdot) \big\|_\mathbb{K} \leq \varepsilon
	\end{equation}
is fulfilled for all $t\geq \tau$, we say the system is \emph{$\varepsilon$-lumpable} with respect to $\xi$.
\end{definition}

In words, lumpability means that for sufficiently large $t$, the transition densities $p^t(x,\cdot)$, i.e., the probabilities of transitions starting from $x$, depend essentially only on the value $\xi(x)$ of the RC at $x$, and not on the precise location of $x$ on the $\xi(x)$-level set of $\xi$.
Condition~\eqref{eq:lumpability} is therefore a sensible definition for a RC that is supposed to describe the effective long-term dynamics of $X_t$.

\begin{remark}
	Our definition of lumpability can be seen as a continuous version of the lumpability condition for discrete Markov chains, originally formulated by Kemeny and Snell~\cite{kemeny_finite_1976}. There it has been shown that the discrete version of lumpability is a necessary and sufficient condition for the Markov chain to be ``compressible'' into a Markov chain between certain \emph{aggregates} of the original chain.
	In general, this compression however leads to a loss of information, i.e.,  restoration of the original chain is in general not possible under the lumpability assumption alone.
\end{remark}

\subsubsection*{Decomposability}
The second condition on $(X_t)$ we will discuss is the following:

\begin{definition}[Decomposability]
If there exists a domain $\Z\subset \R^r$, a function $\xi\in C(\X,\Z)$, a family of time-parametrized functions  $p_D^t:\X\times \Z\rightarrow \R^+$ and a lag time $\tau>0$ such that
\begin{equation}
\label{eq:decomposability}
\tag{\text{D}}
\frac{1}{|\Z|} \big\| p^t(\ast,\cdot) - p_D^t(\ast,\xi(\cdot))\pi(\cdot) \big\|_\mathbb{K}\leq \varepsilon
\end{equation}
is fulfilled for $t\geq \tau$, we say the system is \emph{$\varepsilon$-decomposable} with respect to $\xi$.
\end{definition}

The value of $p^t(x,y)$ can be interpreted as the transition probability from an infinitesimal neighborhood around $x$ to an infinitesimal neighborhood around $y$. Decomposability now holds, among others, for systems for which this transition can be decomposed into two consecutive stages:
\begin{enumerate}
	\item[1)] a slow transition from $x$ to anywhere on the level set $\Sigma_\xi\big(\xi(y)\big)$, for which the (infinitesimal) transition probability is given by $p^t_D(x,\xi(y))$, and
	\item[2)] an instantaneous equilibration on that level set with respect to the invariant density $\pi(y)$, which does no longer depend on the starting point.
\end{enumerate}
Hence, for systems for which this decomposition exists, condition~\eqref{eq:decomposability} again describes a sensible definition of good RCs.
Note however that not only such ``physically decomposable'' systems fulfil the above definition, but that~\eqref{eq:decomposability} describes a more generic ``conceptual decomposability'', as we will see in Section~\ref{sec:examples_reducible}.

\begin{remark}
The above notion of decomposability also has a time- and space-discrete equivalent for discrete Markov chains, which was defined by some of the authors in~\cite{bittracher_probabilistic_2021} (however the condition was called ``deflatability'' there). There it was shown that for Markov chains fulfilling both the lumpability and deflatability condition, a ``compressed'' Markov chain between certain aggregates of the original states exists, and that the full Markov chain can be restored from the compressed chain.
\end{remark}

\begin{remark}
\label{rem:trivial_lumpability_decomposability}
It should be noted that, for large enough lag time, every uniformly ergodic system is
trivially lumpable and decomposable since
\begin{equation}
 \sup_{x \in \X}	\norm{ P^t(x, \cdot) - \mu(\cdot)  }_{TV} =
 \sup_{x \in \X}	\frac{1}{2} \norm{ p^t(x, \cdot) - \pi(\cdot)  }_{L^1} \to 0
\end{equation} as $t \to \infty$
(see for example \cite[Section 3.3 together with Proposition 3(f)]{roberts_general_2004}. Hence, choosing $p_L^t(z,\cdot) = \pi$ and $p_D^t(z,\cdot)=1$ in the above definitions will give lumpability and decomposability with respect to any constant RC since then $p_L^t$ and $p_D^t$ are independent of $x$.
Likewise, every system is trivially lumpable and decomposable, for any tolerance and lag time, with respect to the trivial $n$-dimensional RC $\xi(x)=x$, since choosing $p_L^t(x,\cdot)=p^t(x,\cdot)$ and $p_D^t(\cdot,y) = p^t(\cdot,y)/\pi(y)$ fulfils~\eqref{eq:lumpability} and~\eqref{eq:decomposability} with $\varepsilon=0$.
We emphasize that in this paper we specifically care about systems which are lumpable/decomposable
with respect to intermediate lag times $\tau$ that are much smaller than the equilibration time scale of the system, as well as small dimensions $r$ and $\varepsilon>0$.

Moreover, a system may be lumpable/decomposable with respect to more than one non-trivial combination of $\varepsilon, \tau$ and $r$. In cases where no clear time scale separation exists in the full system, a balance has to be struck between the achievable approximation error of a reduced model built using $\xi$ (acceptable $\varepsilon$), the time scale above which the reduced model is valid (choice of $\tau$), and the dimension of the reduced model (choice of $r$).

Although an interesting objective especially in systems with cascades of time scales, optimizing the choice of $r$ and $\tau$ are not subject of this paper. We will later consider $r$ and $\tau$ to be fixed, and search for corresponding ``optimal'' RCs, i.e. an $r$-dimensional $\xi$ for which~\eqref{eq:lumpability} and/or~\eqref{eq:decomposability} are fulfilled for the smallest possible~$\varepsilon$ (more on that in Section~\ref{sec:variational}).
\end{remark}

\subsubsection*{Connection to Reversibility}
As mentioned above, the two conditions~\eqref{eq:lumpability} and~\eqref{eq:decomposability} are equivalent in reversible systems:

\begin{proposition}
\label{prop:equivalence_reversible_epsilon}
	Let the system be reversible, i.e., let~\eqref{eq:detailed_balance} hold. Then the system is $\varepsilon$-lumpable if and only if it is $\varepsilon$-decomposable.
\end{proposition}
\begin{proof}
Let~\eqref{eq:lumpability} hold for some family of functions $p_L^t:\Z\times \X\rightarrow \R$. Define the family of functions $p_D^t:\X\times \Z \rightarrow \R$ by
$$
p_D^t(x,z) := \frac{p_L^t(z,x)}{\pi(x)}.
$$
Then
\begin{align*}
	\| p_D^\tau(x,\xi(\cdot))\pi(\cdot) - p^\tau(x,\cdot) \|_{L^1} &= \| p_L^\tau(\xi(\cdot),x)\frac{\pi(\cdot)}{\pi(x)} - p^\tau(x,\cdot) \|_{L^1} \\
	&\overset{\eqref{eq:detailed_balance}}{=} \| p_L^\tau(\xi(\cdot),x)\frac{\pi(\cdot)}{\pi(x)} - p^\tau(\cdot,x)\frac{\pi(\cdot)}{\pi(x)} \|_{L^1}\\
	&= \|p_L^\tau(\xi(\cdot),x) - p^\tau(\cdot,x) \|_{L^1_\mu} \pi(x)^{-1}.
\end{align*}
Hence, with this $p^t_D$, it holds
\begin{align*}
	\frac{1}{|\Z|} \big\| p^\tau(\ast,\cdot) - p^\tau_D(\ast,\xi(\cdot))\pi(\cdot) \big\|_\mathbb{K}  &= \frac{1}{|\Z|} \int_\X \big\| p^\tau(x,\cdot) - p^\tau_D(x,\xi(\cdot))\pi(\cdot) \big\|_{L^1} \ts d\mu(x) \\
	&= \frac{1}{|\Z|} \int_\X \big\| p_L^\tau(\xi(x),\cdot) - p^\tau(x,\cdot) \big\|_{L^1_\mu} \pi(x)^{-1} \ts d\mu(x) \\
	&= \frac{1}{|\Z|} \big\| p_L^\tau(\xi(\ast),\cdot) - p^\tau(\ast,\cdot) \big\|_\mathbb{K} \leq \varepsilon.
\end{align*}

For the reverse direction, let~\eqref{eq:decomposability} hold for some family of functions $p_D^t:\X\times \Z\rightarrow \R$ and define $p_L^t: \Z\times \X\rightarrow\R$ by
$$
p_L^t(z,y):= p_D^t(y,z) \pi(y).
$$
We then obtain $\varepsilon$-lumpability with respect to $\xi$ and this $p^t_L$ by performing the above transformations in reverse.
\end{proof}

\begin{remark}
	Proposition~\ref{prop:equivalence_reversible_epsilon} implies that, whenever a system is $\varepsilon$-lumpable or $\varepsilon$-decomposable, there exists a reduced transition kernel $\tilde{p}^t:\Z\times\Z \rightarrow \R^+$, such that
	$$
	\big\| p^t(\ast,\cdot) - \tilde{p}^t(\xi(\ast),\xi(\cdot)) \pi(\cdot) \big\|_\mathbb{K} \leq \varepsilon
	$$
	for $t\geq \tau$. Under this condition, knowing the triple $(\xi, \tilde{p}^t,\pi)$ allows us to approximately re-construct the effective long-term dynamics of the full system. This represents the theoretical justification of our definition of a good RC. 
\end{remark}

\subsection{Examples of lumpable and decomposable systems}
\label{sec:examples_reducible}

An obvious question is whether the conditions~\eqref{eq:lumpability} and~\eqref{eq:decomposability} are consistent with established concepts of reducible systems.
We hence now present several systems with known good RCs and show that they are indeed either lumpable or decomposable.

\subsubsection{Existence of a transition manifold}

Lumpability is strongly connected to the concept of the so-called \emph{transition manifold}, which was introduced a few years ago by some of the authors in order to formulate a geometrical approach to the computation of optimal RCs~\cite{bittracher_transition_2017}.

\begin{definition}[Reducibility and Transition Manifold]
	For $\varepsilon>0, r\leq n, \tau\in \R^+_0$, we call the system \emph{$(\varepsilon,r,\tau)$-reducible} if there exists an $r$-dimensionally parametrizable manifold~$\M\subset \{p^\tau(x,\cdot),~x\in\X\}$
so that for all $x\in\X$
\begin{equation}
\label{eq:reducible}
\big\| \Q\left( x \right) - p^\tau(x,\cdot) \big\|_{L^2_{1/\pi}} \leq \varepsilon,
\end{equation}
where $\Q: \X \rightarrow \M$ is the nearest point projection onto $\M$,
\begin{equation}
\label{eq:Qprojection}
	\Q\left( x \right) := \argmin_{p\in\M}\big\|p^\tau(x,\cdot) - p\big\|_{L^2_{1/\pi}}.
\end{equation}
We call any $\M$ that fulfills~\eqref{eq:reducible} a \emph{transition manifold} of the system.
\end{definition}

The intuition behind~\eqref{eq:reducible} is that the set of all transition densities $\{p^t(x,\cdot),~x\in\X\}$ clusters $\varepsilon$-closely around an $r$-dimensional manifold $\M$ with respect to the $L^2_{1/\pi}$ norm. The density $\Q(x)$ can be understood as the reduced density $p^\tau_L(\xi(x)\cdot)$ in~\eqref{eq:lumpability}. Let $\mathcal{E}:\M\rightarrow \R^r$ be any parametrization of $\M$. It can be shown that the \emph{transition manifold RC}
\begin{equation}
\label{eq:TMRC}
\xi(x):= \mathcal{E}\left(\mathcal{Q}(x)\right)
\end{equation}
is a good RC, in the sense that the projection error of the leading transfer operator eigenfunctions onto $\xi$ is at most $\varepsilon$~\cite{bittracher_transition_2017}.

The computational strategy behind the transition manifold approach now is to sample the set $\{p^t(x,\cdot),~x\in\X\}$ (for example by randomly selecting starting points $x_m\in\X,~m=1,2,\ldots$ and estimating the $p^\tau(x_m,\cdot)$ by parallel simulation), and applying an \emph{unsupervised manifold learning method} (such as diffusion maps~\cite{coifman_diffusion_2006}) to the samples. This strategy has been successfully applied to multiple high-dimensional molecular systems and was confirmed to produce physically interpretable RCs~\cite{bittracher_data-driven_2018, bittracher_dimensionality_2020}.

The concept of the transition manifold was recently re-visited and extended to a broader class of dynamical systems~\cite{bittracher_weak_2020} (the central object now being called~\emph{weak transition manifold}), by requiring the closeness to the manifold now only averaged over the level sets of $\mathcal{Q}$. It however requires the restriction to the following class of ``smooth'' transition manifolds:

\begin{definition}[Smooth manifold]
\label{def:smooth_manifold}
	We call an $r$-dimensionally parametrizable manifold $\M\subset \{p^\tau(x,\cdot),~x\in\X\}$ \emph{smooth}, if for its nearest point projection $\Q$ holds
	\begin{enumerate}
		\item $\Q$ is continuously differentiable,
		\item for each $p\in\M$, the level set
			$$
			\Sigma_\Q(p) = \{x\in\X~|~\Q(x)=p\}
			$$
			is an $n-r$-dimensional topological submanifold of $\X$.
	\end{enumerate}
\end{definition}

With that we can introduce the weak transition manifold:

\begin{definition}[Weak reducibility and weak transition manifold]
	For $\varepsilon>0, r\leq n, \tau\in \R^+_0$, we call the system \emph{weakly $(\varepsilon,r,\tau)$-reducible} if there exists an $r$-dimensionally parametrizable smooth manifold~$\M\subset \{p^\tau(x,\cdot),~x\in\X\}$
so that for all $x\in\X$
\begin{equation}
\label{eq:weakly_reducible}
\int_{\Sigma_\Q(\Q(x))} \big\| \Q\left( x \right) - p^\tau(x',\cdot) \big\|_{L^2_{1/\pi}} \ts d\mu_{\Q(x)}(x') \leq \varepsilon,
\end{equation}
where $\mu_\mathcal{Q}$ is the marginal invariant measure on $\Sigma_\Q(p)$, defined analogously to~\eqref{eq:def_muz}. We call any smooth manifold $\M$ that fulfils~\eqref{eq:weakly_reducible} a \emph{weak transition manifold}.
\end{definition}
It is easy to confirm that every reducible system with a smooth transition manifold is weakly reducible.

The closeness condition for the weak transition manifold~\eqref{eq:weakly_reducible} is similar in nature to the lumpability condition~\eqref{eq:lumpability}, in that the original transition density functions, $p^\tau(x,\cdot)$, are close so some projected or reduced transition density function, $\mathcal{Q}\left(p^\tau(x,\cdot)\right)$ and $p_L(\xi(x),\cdot)$, respectively. Indeed, we will now show that a system which is weakly $(\varepsilon,r,\tau)$-reducible is $\varepsilon$-lumpable with respect to the transition manifold RC $\xi$. Lumpability, and equivalently decomposability, are therefore generalizations of weak reducibility. 

\begin{proposition}
	\label{prop:lumpability}
	Let the system be weakly $(\varepsilon,r,\tau)$-reducible with transition manifold $\mathbb{M}$. Assume that there exists a parametrization $\mathcal{E}:\M\rightarrow\Z\subset R^r$ such that $\xi$ defined by~\eqref{eq:TMRC} is a smooth RC (see Definition~\ref{def:smooth_rc}). Then there exists a family of functions $p_L^t:\Z\times \X\rightarrow \R^+$ such that~\eqref{eq:lumpability} is fulfilled with respect to $\xi$.
\end{proposition}
\begin{proof}
Let $\mathcal{E}:\M\rightarrow\R^r$ be any parametrization of the transition manifold $\M$ and let $\Z := \mathcal{E}(\M)$. In particular, $\mathcal{E}:\M\rightarrow\Z$ is one-to-one. Define the RC $\xi$ by~\eqref{eq:TMRC} and the reduced density $p_L^\tau$ by
	$$
	p_L^\tau(z,\cdot): = \mathcal{E}^{-1}(z).
	$$
Then for any $z\in\Z$ there exists an $x\in\X$ such that $z=\mathcal{E}\big(\Q(x)\big)$ and hence, due to $\mathcal{E}$ being one-to-one, $\Sigma_\xi(z) = \Sigma_{\Q}\big(\Q(x)\big)$. For this $x$ it holds
\begin{align*}
	\int_{\Sigma_\xi(z)} \big\| p_L^\tau(z,\cdot) - p^\tau(x',\cdot) \big\|_{L^1} \ts d\mu_{z}(x') &= \int_{\Sigma_\Q(\Q(x))} \big\| p_L^\tau\big( \mathcal{E}(\Q(x)),\cdot) - p^\tau(x',\cdot) \big\|_{L^1}\ts d\mu_{\Q(x)}(x')\\
	&= \int_{\Sigma_\Q(\Q(x))} \big\| \mathcal{E}^{-1}\big( \mathcal{E}(\Q(x))\big) - p^\tau(x',\cdot) \big\|_{L^1}\ts d\mu_{\Q(x)}(x')\\
	&= \int_{\Sigma_\Q(\Q(x))} \big\| \Q(x) - p^\tau(x',\cdot) \big\|_{L^1}\ts d\mu_{\Q(x)}(x').
\end{align*}
By application of the co-area formula, and the definition~\eqref{eq:def_muz} of $\mu_z$ we get
\begin{align*}
\frac{1}{|\Z|}\big\|p^\tau(\ast,\cdot) - p_L^\tau(\xi(\ast),\cdot)\big\|_\mathbb{K} &= \frac{1}{|\Z|} \int_\Z\int_{\Sigma_\xi(z)} \big\| p_L^\tau(z,\cdot) - p^\tau(x',\cdot) \big\|_{L^1} \ts d\mu_{z}(x')\ts dz,
\intertext{which finally, with $\|f\|_{L^1} = \|f/\pi\|_{L^1_\mu} \leq \|f/\pi\|_{L^2_\mu} = \|f\|_{L^2_{1/\pi}}$, becomes}
	&\leq \frac{1}{|\Z|} \int_\Z \int_{\Sigma_\Q(\Q(x))} \big\| \Q(x) - p^\tau(x',\cdot) \big\|_{L^2_{1/\pi}}\ts d\mu_{\Q(x)}(x')\ts dz\\
&\leq \sup_{z\in\Z} \int_{\Sigma_\Q(\Q(x))} \big\| \Q(x) - p^\tau(x',\cdot) \big\|_{L^2_{1/\pi}}\ts d\mu_{\Q(x)}(x')\\
&\overset{\eqref{eq:weakly_reducible}}{\leq} \varepsilon.
\end{align*} \qedhere
\end{proof}

\subsubsection{Slow- and fast components}
\label{sec:slowfast}


Next, we show that a process defined by an SDE with slow and fast components is decomposable with respect to the slow component.
We utilize a multiscale decomposition of the corresponding infinitesimal generator, together with a multiscale ansatz for the transition densities $p^t(x,\cdot)$. In that we utilize well-known averaging techniques from~\cite{pavliotis_multiscale_2008}.

It will prove advantageous to consider the transition densities as densities with respect to the stationary density . That is, we define for each $x\in\X$ the density $q^t(x,\cdot)$ with respect to $\pi$ by
$$
p^t(x,\cdot) = q^t(x,\cdot)\pi(\cdot).
$$

Let the components of $(X_t)$ be dividable into two processes $(Y_t)$ on $\Y$, $(Z_t)$ on $\Z$, such that $\X = \Y \oplus \Z$,
$$
(X_t) = \begin{pmatrix}Y_t\\ Z_t\end{pmatrix},
$$
and the components fulfil the system of SDEs
\begin{equation}
	\label{eq:smoluchowski_divided}
	\begin{aligned}
		\varepsilon d Y_t &= -\nabla_y V(Y_t,Z_t)\ts dt + \sqrt{\varepsilon}\sigma d W^\Y_t\\
		 d Z_t &= -\nabla_z V(Y_t,Z_t)\ts dt + \sigma d W^\Z_t,
	\end{aligned}
\end{equation}
with potential $V:\X\rightarrow \R$, scalar diffusion parameter $\sigma>0$ and $(W^\Y_t)$, $(W^\Z_t)$ standard Wiener processes on $\Y$ and $\Z$, respectively. A timescale separation parameter $0 <\varepsilon \ll 1$ ensures that $(Y_t)$ evolves ``fast'' compared to $(Z_t)$.

The evolution of $q^t$ under $(X_t)$ is now governed by the Fokker Planck equation
$$
\partial_t q^t(x,\cdot) = \mathcal{L}_\varepsilon q^t(x,\cdot),\quad q^0(x,\cdot) = \delta_x(\cdot),
$$
where the infinitesimal generator $\mathcal{L}_\varepsilon$ has the multiscale structure
\begin{align}
	\label{eq:Leps_Smolu}	\mathcal{L}_\varepsilon &= \frac{1}{\varepsilon} \mathcal{L}_y + \mathcal{L}_z,\\
	\intertext{with the two compontents}
	\notag	\mathcal{L}_y &= \frac{\sigma^2}{2} \Delta_y - \nabla_y V\cdot \nabla_y, \\
	\notag	\mathcal{L}_z &= \frac{\sigma^2}{2} \Delta_z - \nabla_z V\cdot \nabla_z.
\end{align}

We want to investigate to what extent $q^t(x,\cdot)$, and by extension $p^t(x,\cdot)$, can be approximated by an ``essential transition density'' that depends only on the slow variable $z$ in the setting $\varepsilon \ll 1$, $t=\mathcal{O}(1)$. To this end, we make the multiscale ansatz for $q^t$
\begin{equation}
	\label{eq:multiscale_ansatz_q_Smolu}
	q^t(x,\cdot) = q_0^t(x,\cdot) + \varepsilon q_1^t(x,\cdot) + \mathcal{O}(\varepsilon^2).
\end{equation}

Inserting~\eqref{eq:multiscale_ansatz_q_Smolu} into~\eqref{eq:Leps_Smolu} gives
\begin{equation}
\label{eq:multiscale_equation_Smolu}
\partial_t q_0^t(x,\cdot) + \varepsilon \partial_t q_1^t(x,\cdot) + \mathcal{O}(\varepsilon^2) = \frac{1}{\varepsilon} \mathcal{L}_y q_0^t(x,\cdot) + \mathcal{L}_y q_1^t(x,\cdot) + \mathcal{L}_z q_0^t(x,\cdot) + \mathcal{O}(\varepsilon).
\end{equation}
Comparing the terms of order $\varepsilon^{-1}$ gives
$$
\mathcal{L}_y q_0^t(x,\cdot) = 0.
$$
By~\cite[Sec.~10.2]{pavliotis_multiscale_2008}, $\mathcal{L}_y$ has a one-dimensional null space consisting of functions constant in $y$.  Hence, $q_0^t(x,\cdot)$ is independent of $y$ in its second argument, and so
$$
q^t\left((y,z),(y',z')\right) = q^t_0\left((y,z),z'\right) + \mathcal{O}(\varepsilon).
$$
Therefore, for $t=\mathcal{O}(1)$, the transition density $p^t(x,\cdot)$ takes the form
$$
p^t\left((y,z),(y',z')\right) = q_0^t((y,z),z')\ts\pi\left((y',z')\right) + g\left((y,z),(y',z')\right)\ts\pi\left((y',z')\right) .
$$
for some function $g\in \mathcal{O}(\varepsilon)$.
Applying the $\|\cdot\|_\mathbb{K}$-norm, it follows that the system is decomposable with respect to the RC $\xi(x)=z$, i.e.,
$$
\frac{1}{|\Z|} \left\|q_0^t(\ast,\xi(\cdot))\pi(\cdot) - p^t(\ast,\cdot)\right\|_\mathbb{K} = \mathcal{O}(\varepsilon).
$$

\markchange{
\begin{remark}
	The above result, or to be precise, the multiscale extension~\eqref{eq:multiscale_ansatz_q_Smolu}, holds only for times $t$ up to $\mathcal{O}(1)$, as one can show that the second term $q_1$ is bounded uniformly only in $t$~\cite{koltai_multiscale_2018}, hence may grow linearly in $t$.
	However, we are only interested in lumpability on moderate lag times anyway, as every system becomes trivially lumpable for $t\rightarrow \infty$ (see Remark~\ref{rem:trivial_lumpability_decomposability}).
\end{remark}
\begin{remark}
	While not necessarily in the scope of this paper, we can continue the multiscale analysis in order to derive an evolution equation for $q^t_0$. See Appendix~\ref{sec:extended_multiscale} for details.
\end{remark}
}

\subsubsection{Generator spectral gap}
\label{sec:spectralgap}

Finally, we show that systems whose infinitesimal generator exhibits a spectral gap of sufficient size, such as metastable systems~\cite{dellnitz_approximation_1999, deuflhard_identification_2000}, are decomposable with respect to some non-trivial RC.

Consider again the transition densities $q^t$ with respect to the stationary measure, and its Fokker Planck equation
$$
\partial_t q^t(x,\cdot) = \mathcal{L} q^t(x,\cdot),\quad q^0(x,\cdot) = \delta_x(\cdot).
$$
We assume that the spectrum of the generator $\mathcal{L}$ is real, non-positive and discrete, and denote the eigenvalues of $\mathcal{L}$ in descending order:
$$
0 = \theta_0 > \theta_1 \geq \theta_2 \geq \ldots ,
$$
repeated by geometric multiplicity. Let furthermore $\varphi_i$ denote the eigenfunction belonging to $\theta_i$. The $\varphi_i$ then form an orthonormal basis of $L^2_\mu(\X)$~\cite{nuske_spectral_2021}. As $q^t(x,\cdot) \in L^1_\mu(\X) \cap L^\infty_\mu(\X)$ for any $x$, $q^t(x,\cdot) \in L^2_\mu(\X)$. We can then describe the evolution of the density $q^t(x,\cdot)$ by
\begin{equation}
\label{eq:qt_evolution_eigenpairs}
q^t(x,\cdot) = \sum_{i=0}^\infty e^{\theta_i t} c_i(x) \varphi_i(\cdot)
\end{equation}
where $c_i:\X\rightarrow \R$ are some functions that do not depend on the $\theta_i$.

Now, we additionally assume that the eigenvalues can be separated into a dominant and a non-dominant part. Specifically, we assume there exists an index $K>0$, so that the ratio
\begin{equation}
\label{eq:theta_ratio}
\frac{\theta_K}{\theta_{K+1}} >0
\end{equation}
is small.
This situation for example occurs if the system is \emph{metastable}, i.e., there exists a partition $\X = \X_1\cup\ldots \cup \X_K$ of state space into disjoint regions, and the system is almost invariant on each $\X_i$ on relatively long time scales.
For a precise introduction of metastability and its connection to the dominant spectrum see~\cite{schutte_metastability_2014}.

Now suppose that there exists an integer $r\leq K$ and a RC $\xi:\X\rightarrow \R^r$ such that $\xi$ parametrizes the dominant $\varphi_i$, i.e., there exist some functions $\tilde{\varphi}_i:\R^r\rightarrow \R$ such that
\begin{equation}
\label{eq:eigenfunction_parametrization}
\varphi_i = \tilde{\varphi}_i\circ \xi \quad i=1,\ldots,K.
\end{equation}
Such a $\xi$ always exists, as one can always choose
$$
r:=K,\quad \xi_i :=\varphi_i\quad \text{and}\quad \tilde{\varphi}_i(z) := z_i.
$$
Often, however, the dominant eigenfunctions possess some common lower-dimensional, non-linear parametrization. For metastable systems, this is for example the case if the metastable sets $\X_1,\ldots,\X_K$ are connected by just a small number of transition pathways. An example system with five metastable sets, but one common transition pathway, hence a one-dimensional RC, can be found in Section~\ref{sec:example_metastable_circular}.

Let $\varepsilon>0$ be some small constant. We now show that, if $t=t(\varepsilon)$ is large enough, and the spectral ratio~\eqref{eq:spectral_gap_lumpability} is small enough, then the system is $\varepsilon$-decomposable with respect to any $\xi$ fulfilling~\eqref{eq:eigenfunction_parametrization}. To see this, split the right hand side of~\eqref{eq:qt_evolution_eigenpairs} into the dominant and the non-dominant part:
\begin{equation*}
\label{eq:qt_evolution_eigenpairs_split}
q^t(x,\cdot) = \sum_{i=0}^K e^{\theta_i t} c_i(x) \varphi_i(\cdot) + \sum_{i=K+1}^\infty e^{\theta_i t} c_i(x) \varphi_i(\cdot) .
\end{equation*}
Due to~\eqref{eq:eigenfunction_parametrization}, the first summand depends only on $\xi$:
$$
\sum_{i=0}^K e^{\theta_i t} c_i(x) \varphi_i(\cdot) = \underbrace{\sum_{i=0}^K e^{\theta_i t} c_i(x) \tilde{\varphi}_i\left(\xi(\cdot)\right)}_{=: p_D^t\left(x,\xi(\cdot)\right)}.
$$
To show that the system is $\varepsilon$-decomposable with respect to $\xi$ and $p_D^t$, we thus have to ensure that
\begin{equation}
\label{eq:spectral_gap_lumpability}
\frac{1}{|\Z|} \left\| \sum_{i=K+1}^\infty e^{\theta_i t} c_i(\ast) \varphi_i(\cdot) \pi(\cdot) \right\|_\mathbb{K} \leq \varepsilon.
\end{equation}
Since the $\theta_i$ are decreasing, it holds
\begin{align*}
	\left\| \sum_{i=K+1}^\infty e^{\theta_i t} c_i(\ast) \varphi_i(\cdot) \pi(\cdot) \right\|_\mathbb{K} &\leq e^{\theta_{K+1} t} \left\|  \sum_{i=K+1}^\infty c_i(\ast) \varphi_i(\cdot) \pi(\cdot) \right\|_\mathbb{K},
	\shortintertext{and further}
	&\leq e^{\theta_{K+1} t} \left( \underbrace{\left\|  \sum_{i=0}^K c_i(\ast) \varphi_i(\cdot) \pi(\cdot) \right\|_\mathbb{K}}_{=:\tilde{C}} + \left\|  \sum_{i=0}^\infty c_i(\ast) \varphi_i(\cdot) \pi(\cdot) \right\|_\mathbb{K}\right).
\end{align*}
The first summand, denoted $\tilde{C}$, is finite as a finite sum. As
$$
p^t(\ast,\cdot) = q^t(\ast,\cdot)\pi(\cdot) = \sum_{i=0}^\infty e^{\theta_i t} c_i(\ast) \varphi_i(\cdot) \pi(\cdot),
$$
and $1=e^{\theta_i 0}$, the second summand can be interpreted as the $\mathbb{K}$-norm of $p^t\big|_{t=0}$:
$$
\left\|  \sum_{i=0}^\infty c_i(\ast) \varphi_i(\cdot) \pi(\cdot) \right\|_\mathbb{K} = \lim_{t\rightarrow 0}\left\| p^t(\ast,\cdot)\right\|_\mathbb{K}.
$$
By writing out the $\mathbb{K}$-norm, it can easily be seen that $\left\| p^t(\ast,\cdot)\right\|_\mathbb{K}=1$ for all $t>0$, and thus
$$
\lim_{t\rightarrow 0}\left\| p^t(\ast,\cdot)\right\|_\mathbb{K} = 1.
$$
With $C:=\frac{\tilde{C}+1}{|\Z|}$, we therefore get
$$
\frac{1}{|\Z|} \left\| \sum_{i=K+1}^\infty e^{\theta_i t} c_i(\ast) \varphi_i(\cdot) \pi(\cdot) \right\|_\mathbb{K} \leq C\ts e^{\theta_{K+1}t}.
$$
Hence, if we choose
\begin{equation}
\label{eq:decomposability_timescale}
t\geq t(\varepsilon):= \frac{1}{\theta_{K+1}} \log\left(\frac{\varepsilon}{C}\right),
\end{equation}
then~\eqref{eq:spectral_gap_lumpability} is fulfilled, and the system is $\varepsilon$-decomposable.

Now, of course, for a decreasing decomposability tolerance ($\varepsilon\rightarrow 0$), we will need to increase the lag time ($t(\varepsilon)\rightarrow \infty$). In this situation, every system becomes trivially decomposable, see Remark~\ref{rem:trivial_lumpability_decomposability}.
In order to claim \emph{non-trivial} decomposability for the lag time $t(\varepsilon)$, we have to ensure that $p^t_D(x,\cdot)$ remains ``expressive'', i.e., is not close to the identity. For this we need to show that the factors $e^{\theta_i t(\varepsilon)}, i=1,\ldots,K$, have not yet decayed substantially. Note that
\begin{align*}
e^{\theta_it(\varepsilon)} &= e^{\frac{\theta_i}{\theta_{K+1}} \log\left(\frac{\varepsilon}{C}\right)}\\
&\geq e^{\frac{\theta_K}{\theta_{K+1}} \log\left(\frac{\varepsilon}{C}\right)} \quad \text{for}\quad i=1,\ldots,K.
\end{align*}
We need to ensure that the spectral ratio $\frac{\theta_K}{\theta_{K+1}}$ decreases sufficiently quickly for $\varepsilon\rightarrow 0$. Assuming 
\begin{equation}
\label{eq:rhoepsilon}
\frac{\theta_K}{\theta_{K+1}} = \mathcal{O}\left( \log(\varepsilon)^{-1} \right) \quad (\varepsilon \rightarrow 0),
\end{equation}
we have
\begin{align*}
e^{\theta_it(\varepsilon)} &= \mathcal{O}(1) \quad (\varepsilon\rightarrow 0),
\shortintertext{and thus also, for all $x\in\X$,}
\left|p_D^{t(\varepsilon)}(x,\cdot) - 1 \right| &= \mathcal{O}(1)\quad (\varepsilon\rightarrow 0).
\end{align*}
\begin{example}
	The assumption \eqref{eq:rhoepsilon} on the spectral ratio is readily fulfilled in systems with a natural spectral gap, such as systems with energy barriers. Consider the overdamped Langevin equation
	\begin{equation}
		\label{eq:metastable_SDE}
		dX_t = -\nabla V(X_t) + \sqrt{2 \varepsilon}\ts dW_t,
	\end{equation}
	describing the motion of a particle in the potential $V$ at temperature $\varepsilon$. 	Now assume that $V$ possesses multiple local minima separated by saddle points, such as the double well potential depicted in Figure~\ref{fig:doublewell}. Transitions out of the regions around the local minima (called \emph{metastable sets}) are induced by the stochastic forcing in~\eqref{eq:metastable_SDE}, with a rate governed by~$\varepsilon$. To be specific, the expected exit rate $\rho$ from one mimimum to another across a barrier of height $\Delta V$ is given by the van't~Hoff-Arrhenius law
	$$
	\rho \sim e^{-\frac{\Delta V}{\varepsilon}},
	$$
	where ``$\sim$'' indicates asymptotic equality in the sense of $\rho = C \exp(-\Delta V/\varepsilon)(1+\mathcal{O}(\sqrt{\varepsilon}\log\varepsilon))$, where $C$ is an $\varepsilon$-independent prefactor \cite{Bovier1}. 
	
	\begin{figure*}
		\centering
		\includegraphics{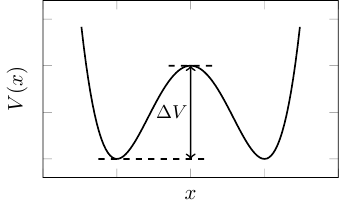}
		\caption{One-dimensional double well potential. At low temperature, transitions accross the central energy barrier are rare. The barrier height $\Delta V$, along with the system temperature, determines the transition rate.}
		\label{fig:doublewell}
	\end{figure*}
	
	Recall that the $i$-th eigenvalue $\theta_i$ is the negative of the equilibration rate of the $i$-th slowest sub-process. At low enough temperature, the metastable transitions are hence associated with the dominant eigenvalues. The equilibration within the wells on the other hand resembles an Ornstein-Uhlenbeck process, whose eigenvalues are independent of the temperature~\cite{klus_data-driven_2020}. There hence exists an index $K\geq 1$ (which depends on the number of local minima and barriers between them) such that \cite{Bovier2}
	\begin{equation*}
	\begin{aligned}
		\theta_i &\sim e^{\frac{\Delta V_i}{\varepsilon}}, & i&=1,\ldots,K, \\
		\theta_i &\sim 1 ,& i &\geq K+1,
	\end{aligned}
	\end{equation*}
	where $\Delta V_i$ denotes the barrier height between the $i$-th pair of local minima.
	For the spectral ratio then holds
	$$
	\frac{\theta_K}{\theta_{K+1}} \sim e^{-\frac{\Delta V_K}{\varepsilon}} .
	$$

For a given small enough temperature $\varepsilon$, the system is now non-trivially $\varepsilon$-decomposable on the timescale $t(\varepsilon)$ predicted by~\eqref{eq:decomposability_timescale}. One one hand, by construction of $t(\varepsilon)$, the fast processes have already decayed.
On the other hand, for the slow timescales holds, for some $C,\tilde{C}$,
\begin{align*}
	e^{\theta_i t(\varepsilon)}
	&= e^{\tilde{C}e^{-\frac{\Delta V_K}{\varepsilon}} \log\left(\frac{\varepsilon}{C}\right)}.
\end{align*}
As one can easily check,
$$
	\lim_{\varepsilon \to 0} e^{\tilde{C}e^{-\frac{\Delta V_K}{\varepsilon}} \log\left(\frac{\varepsilon}{C}\right)} = 1,
$$
so for small enough $\varepsilon$, the slow processes have not decayed yet on timescale $t(\varepsilon)$. 
Hence, the ``natural'' spectral ratio $e^{-\frac{\Delta V}{\varepsilon}}$ has a faster decay rate in $\varepsilon$ than what is required by assumption~\eqref{eq:rhoepsilon}.

\end{example}

\section{A variational principle for optimal RCs}
\label{sec:variational}

From here on, we consider the reduced dimension $r\leq n$ as well the lag time $\tau$ to be predetermined and fixed.
For simplicity of notation, we will omit the lag time as parameter for the transition kernel, i.e., define~$p(\cdot,\cdot):= p^\tau(\cdot,\cdot)$,~$p_L(\cdot,\cdot):=p_L^\tau(\cdot,\cdot)$,~$p_D(\cdot,\cdot):=p_D^\tau(\cdot,\cdot)$.

Our goal is now to find an optimal RC, i.e., a function $\xi:\X\rightarrow \Z\subset \R^r$ that fulfills~\eqref{eq:lumpability} or equivalently~\eqref{eq:decomposability} for the smallest possible $\varepsilon \geq 0$.
Hence formally, we seek the minimizers of the following loss functions:

\begin{definition}[Lumpability and decomposability loss function]
The nonlinear functional $\LL:C(\X,\Z) \rightarrow \R^+$, defined by
	\begin{equation}
	\label{eq:LL}
	\LL(\vartheta) := \frac{1}{|\Z|} \min_{p_L:\Z\times\X\rightarrow \R^+} \|p(\ast,\cdot) - p_L(\vartheta(\ast),\cdot) \|_\mathbb{K}
	\end{equation}
is called the \emph{lumpability loss function} of the system.

The nonlinear functional $\LD:C(\X,\Z) \rightarrow \R^+$, defined by
	\begin{equation}
	\label{eq:LD}
	\LD(\vartheta) := \frac{1}{|\Z|} \min_{p_D:\X\times\Z\rightarrow \R^+} \| p(\ast,\cdot) - p_D(\ast,\vartheta(\cdot)) \pi(\cdot) \|_\mathbb{K}
	\end{equation}
is called the \emph{decomposability loss function} of the system.
\end{definition}

From the equivalence of lumpability and decomposability (Proposition~\ref{prop:equivalence_reversible_epsilon}) it follows that for every $\vartheta\in C(\X,\Z)$ holds
\begin{equation}
\label{eq:equivalence_loss_functions}
\LL(\vartheta) = \LD(\vartheta),
\end{equation}
hence we can find the optimal RC with respect to both~\eqref{eq:lumpability} and~\eqref{eq:decomposability} by solving
\begin{equation}
\label{eq:minimizer_LL}
\xi := \argmin_{\vartheta\in C(\X,\Z)} \LL(\vartheta),
\end{equation}
at least in theory.
Evaluating, let alone minimizing $\LL$ (or $\LD$) would however prove difficult, due to the minimization over an infinite-dimensional function space involved in its definition (the search for the functions $p_L$ or $p_D$ in every step). In the following section, we therefore derive \emph{essentially} equivalent reformulations of $\LL$ and $\LD$ that do not involve this minimization.

\subsection{Differential formulation of lumpability and decomposability}

The condition~\eqref{eq:lumpability} can be interpreted as the closeness of the transition densities $p(x,
\cdot)$ to some reduced reference density $p_L(\xi(x),\cdot)$. This implies that all densities $p(x,\cdot)$ whose starting points $x$ lie on one level set of $\xi$ are close \emph{to each other}.
Likewise, condition~\eqref{eq:decomposability} can be seen as the closeness of the transition observables $p(\cdot,y)$ to some reduced reference observable $p_D(\cdot,\xi(y))\pi(y)$, which implies that all observables $p(\cdot,y)$ whose end points $y$ lie on one level set of $\xi$ are close \emph{to each other}. This observation motivates the following ``differential'' characterization of lumpability and decomposability, for which we now require the smoothness assumptions on $\xi$ and its level sets from Definition~\ref{def:smooth_rc}.

\begin{definition}[Differential lumpability]
If there exists a domain $\Z\subset \R^r$ and a smooth RC $\xi\in S(\X,\Z)$ such that
	\begin{equation}
    	\label{eq:differential_lumpability}
    	\tag{L'}
        \frac{1}{|\Z|}\int_\Z \int_{\Sigma_\xi(z)} \int_{\Sigma_\xi(z)} \norm{ p(x^{(1)}, \cdot) -  p(x^{(2)}, \cdot) }_{L^1 }\ts d\mu_z\big(x^{(1)}\big) \ts d\mu_z\big(x^{(2)}\big)\ts dz \leq \varepsilon
    \end{equation}
is fulfilled, we say the system is \emph{differentially $\varepsilon$-lumpable} with respect to $\xi$.
\end{definition}

\begin{definition}[Differential decomposability]

If there exists a domain $\Z\subset \R^r$ and a smooth RC $\xi\in S(\X,\Z)$ such that
    \begin{equation}
    	\label{eq:differential_decomposability}
    	\tag{D'}
        \frac{1}{|\Z|}\int_\Z \int_{\Sigma_\xi(z)} \int_{\Sigma_\xi(z)} \norm{ p(\cdot,y^{(1)})/\pi(y^{(1)}) -  p(\cdot,y^{(2)})/\pi(y^{(2)}) }_{L^1_\mu}\ts d\mu_z\big(x^{(1)}\big) \ts d\mu_z\big(x^{(2)}\big)\ts dz \leq \varepsilon.
    \end{equation}
is fulfilled, we say the system is \emph{differentially $\varepsilon$ decomposable} with respect to $\xi$.
\end{definition}

As one easily sees using a triangle inequality argument, the differential conditions imply the original conditions, and the original conditions \emph{almost} imply the differential conditions:

\begin{lemma}
\label{lem:differential-epsilon-lumpability}
    If the system is $\varepsilon$-lumpable with respect to $\xi$, then the system is differentially $2\varepsilon$-lumpable with respect to $\xi$.
    \par
    Conversely, if the system is differentially $\varepsilon$-lumpable with respect to $\xi$, then the system is $\varepsilon$-lumpable with respect to $\xi$.
\end{lemma}

\begin{proof}
    Assume that the system is $\varepsilon$-lumpable, i.e., \eqref{eq:lumpability} holds for some function $p_L$. Then we have
   \begin{align*}
        \frac{1}{|\Z|}\int_\Z \int_{\Sigma_\xi(z)} \int_{\Sigma_\xi(z)} &\norm{ p(x^{(1)}, \cdot) -  p(x^{(2)}, \cdot) }_{L^1 }\ts d\mu_z\big(x^{(1)}\big) \ts d\mu_z\big(x^{(2)}\big)\ts dz \\
        &= \frac{1}{|\Z|}\int_\Z \int_{\Sigma_\xi(z)} \int_{\Sigma_\xi(z)} \big\| p(x^{(1)}, \cdot) - \underbrace{p_L(z,\cdot)}_{=p_L(\xi(x^{(1)}),\cdot)} + \underbrace{p_L(z,\cdot)}_{=p_L(\xi(x^{(2)}),\cdot)} -  p(x^{(2)}), \cdot) \Big\|_{L^1 }\ts d\mu_z\big(x^{(1)}\big) \ts d\mu_z\big(x^{(2)}\big)\ts dz \\
        &\leq \frac{1}{|\Z|}\int_\Z \int_{\Sigma_\xi(z)} \big\| p(x^{(1)}, \cdot) - p_L(\xi(x^{(1)}),\cdot) \big\|_{L^1} \ts d\mu_z(x^{(1)})\ts dz + \frac{1}{|\Z|}\int_\Z \int_{\Sigma_\xi(z)} \big\| p(x^{(2)}, \cdot) - p_L(\xi(x^{(2)}),\cdot) \big\|_{L^1} \ts d\mu_z(x^{(2)})\ts dz \\
        &\leq  2 \varepsilon.
   \end{align*}

   For the reverse statement, assume that~\eqref{eq:differential_lumpability} holds. Define
   $$
   p_L(z,\cdot) := \int_{\Sigma_\xi(z)} p(x',\cdot)\ts d\mu_z(x').
   $$
   This $p_L$ exists because $p(\cdot,y)\in L^1_\mu$ for all $y\in\X$.

	Recall that all $\mu_z$ are probability measures on the respective $\Sigma_\xi(z)$. Then
   \begin{align*}
   	\frac{1}{|\Z|}\int_\Z \int_{\Sigma_\xi(z)} &\big\| p_L(z,\cdot) - p(x,\cdot) \big\|_{L^1} \ts d\mu_{z}(x)\ts dz \\
   		&= \frac{1}{|\Z|}\int_\Z \int_{\Sigma_\xi(z)} \Big\| \int_{\Sigma_\xi(z)} p(x',y) - p(x,y) \ts d\mu_z(x') \Big\|_{L^1}\ts d\mu_z(x)\ts dz \\
   		&\leq \frac{1}{|\Z|}\int_\Z \int_{\Sigma_\xi(z)} \int_{\Sigma_\xi(z)} \big\| p(x',\cdot) - p(x,\cdot) \big\|_{L^1}\ts d\mu_z(x')\ts d\mu_z(x)\ts dz\\
   		&\overset{\eqref{eq:differential_lumpability}}{\leq} \varepsilon.
   \end{align*}
\end{proof}

\begin{lemma}
\label{lem:differential-epsilon-decomposability}
    If the system is $\varepsilon$-decomposable
    with respect to $\xi$, then the system is differentially $2\varepsilon$-decomposable with respect to $\xi$.
    \par
    Conversely, if the system is differentially $2\varepsilon$-decomposable with respect to $\xi$, then the system is $\varepsilon$-decomposable with respect to $\xi$.
\end{lemma}

\begin{proof}
	Assume that the system is $\varepsilon$-decomposable, i.e., \eqref{eq:decomposability} holds for some function $p_D$. Then we have
	\begin{align*}
        \frac{1}{|\Z|}\int_\Z \int_{\Sigma_\xi(z)} &\int_{\Sigma_\xi(z)} \norm{ p(\cdot,y^{(1)})/\pi(y^{(1)}) -  p(\cdot,y^{(2)})/\pi(y^{(2)}) }_{L^1_\mu}\ts d\mu_z\big(y^{(1)}\big) \ts d\mu_z\big(y^{(2)}\big)\ts dz \\
        &= \frac{1}{|\Z|}\int_\Z \int_{\Sigma_\xi(z)} \int_{\Sigma_\xi(z)} \norm{ p(\cdot,y^{(1)})/\pi(y^{(1)}) - p_D(\cdot,z) + p_D(\cdot,z) - p(\cdot,y^{(2)})/\pi(y^{(2)}) }_{L^1_\mu}\ts d\mu_z\big(y^{(1)}\big) \ts d\mu_z\big(y^{(2)}\big)\ts dz \\
        &\leq \frac{1}{|\Z|}\int_\Z \int_{\Sigma_\xi(z)} \norm{ p(\cdot,y^{(1)})/\pi(y^{(1)}) - p_D(\cdot,\xi(y^{(1)}))}_{L^1_\mu}\ts d\mu_z\big(y^{(1)}\big)\ts dz \\
        &\qquad + \frac{1}{|\Z|}\int_\Z \int_{\Sigma_\xi(z)} \norm{ p(\cdot,y^{(2)})/\pi(y^{(2)}) - p_D(\cdot,\xi(y^{(2)}))}_{L^1_\mu}\ts d\mu_z\big(y^{(2)}\big)\ts dz\\
        &= \frac{2}{|\Z|}\int_\X \norm{ p(\cdot,y) - p_D(\cdot,\xi(y))\pi(y)}_{L^1_\mu}\ts dy \\
        &= \frac{2}{|\Z|} \big\| \| p(\ast,\cdot) - p_D(\ast,\cdot)\pi(\cdot) \|_{L^1} \big\|_{L^1_\mu} \overset{\eqref{eq:decomposability}}{\leq}  2 \varepsilon.
   \end{align*}

   For the reverse statement, define
   $$
   p_D(\cdot,z) := \int_{\Sigma_\xi(z)} p(\cdot,y')/\pi(y')\ts d\mu_z(y').
   $$
This $p_D$ exists because $p(x,\cdot)\in L^1$ for all $x\in\X$. Then
	\begin{align*}
	\frac{1}{|\mathbb{Z}|} \big\| p(\ast,\cdot) - p_D(\ast,\xi(\cdot))\pi(\cdot) \big\|_\mathbb{K} &= \frac{1}{|\Z|}\big\|  \| p(\ast,\cdot) - p_D(\ast,\xi(\cdot))\pi(\cdot)\|_{L^1} \big\|_{L^1_\mu}\\
	&=	\frac{1}{|\Z|}\int_\Z \int_{\Sigma_\xi(z)} \Big\|p_D(\cdot,z) - p(\cdot,y)/\pi(y) \Big\|_{L^1_\mu}\ts d\mu_z(y)\ts dz \\
   		& = \frac{1}{|\Z|}\int_\Z \int_{\Sigma_\xi(z)} \int_{\X} \Big| \int_{\Sigma_\xi(z)} p(x,y')/\pi(y') - p(x,y)/\pi(y) d\mu_z(y')\ts \Big|\ts d\mu(x)\ts d\mu_z(y)\ts dz \\
   		& \leq \frac{1}{|\Z|}\int_\Z \int_{\Sigma_\xi(z)} \int_{\Sigma_\xi(z)} \big\| p(\cdot,y')/\pi(y') - p(\cdot,y)/\pi(y) \big\|_{L^1_\mu}\ts d\mu_z(y') d\mu_z(y)\ts dz \\
   		&\overset{\eqref{eq:differential_decomposability}}{\leq} \varepsilon.
   \end{align*}
\end{proof}

\subsection{Differential loss functions}

We can now define \emph{differential} versions of the loss functions $\LL$ and $\LD$ in which the hard-to-identify terms $p_L$ and $p_D$ no longer appear:

\begin{definition}[differential lumpability and differential decomposability loss function]
The nonlinear functional $\LRL:S(\X,\Z) \rightarrow \R^+$, defined by
	\begin{equation}
	\label{eq:LVL}
	\LRL(\vartheta) := \frac{1}{|\Z|}\int_\Z \int_{\Sigma_\vartheta(z)} \int_{\Sigma_\vartheta(z)} \norm{ p(x^{(1),\cdot}) -  p(x^{(2)},\cdot) }_{L^1}\ts d\mu_z\big(x^{(1)}\big) \ts d\mu_z\big(x^{(2)}\big)\ts dz
	\end{equation}
is called the \emph{differential lumpability loss function} of the system.

The nonlinear functional $\LRD:S(\X,\Z) \rightarrow \R^+$, defined by
	\begin{equation}
	\label{eq:LVD}
	\LRD(\vartheta) := \frac{1}{|\Z|}\int_\Z \int_{\Sigma_\vartheta(z)} \int_{\Sigma_\vartheta(z)} \norm{ p(\cdot,y^{(1)})/\pi(y^{(1)}) -  p(\cdot,y^{(2)})/\pi(y^{(2)}) }_{L^1_\mu}\ts d\mu_z\big(x^{(1)}\big) \ts d\mu_z\big(x^{(2)}\big)\ts dz
	\end{equation}
is called the \emph{differential decomposability loss function} of the system.
\end{definition}

Note that, unlike $\LL$ and $\LD$, $\LRL$ and $\LRD$ are in general neither identical to each other, nor to $\LL$. The following result characterizes the relationship between the different loss functions:

\begin{corollary}
	\label{cor:comparison loss functions}
	For any $\vartheta\in S(\X,\Z)$ holds
	\begin{align*}
		\LL(\vartheta) &\leq \LRL(\vartheta) \leq 2\ts \LL(\vartheta), \\
	\shortintertext{and}
		\LL(\vartheta) &\leq \LRD(\vartheta) \leq 2\ts \LL(\vartheta).
	\end{align*}
\end{corollary}
\begin{proof}
	The first pair of inequalities follows directly from Lemma~\ref{lem:differential-epsilon-lumpability}. The second pair of inequalities follows from Lemma~\ref{lem:differential-epsilon-decomposability} and~\eqref{eq:equivalence_loss_functions}.
\end{proof}

Hence, for arbitrary (non-optimal) RCs $\vartheta$, we cannot expect $\LRL(\vartheta)$ or $\LRD(\vartheta)$ to be similar to $\LL(\vartheta)$. However, under the assumption that the system is indeed $\varepsilon$-lumpable for small $\varepsilon$, we can expect their \emph{minima} to be similar:

\begin{corollary}
	Let $\xi$ be an optimal RC, defined by~\eqref{eq:minimizer_LL}, and set $\varepsilon:=\mathscr{L}_L(\xi)$. Let $\xi_L$ and $\xi_D$ be minimizers of $\LRL$ and $\LRD$, respectively. Then
	\begin{align*}
		\LRL(\xi) \leq 2\varepsilon \qquad\text{and}\qquad 	\LRD(\xi) \leq 2\varepsilon.
	\end{align*}
	Moreover,
	\begin{align*}
		\LL(\xi_L) \leq 2\varepsilon \qquad\text{and}\qquad \LL(\xi_D) \leq 2\varepsilon.
	\end{align*}
\end{corollary}

\begin{proof}
We show the assertions for $\LRL$. For $\LRD$, the proof is identical.

The first inequality follows from applying Corollary~\ref{cor:comparison loss functions} to $\vartheta=\xi$. The second inequality then follows from
$$
\LL(\xi_L) \overset{\text{Cor.}~\ref{cor:comparison loss functions}}{\leq} \LRL(\xi_L) \overset{\text{Def.}~\xi_L}{\leq} \LRL(\xi) \leq 2\varepsilon.
$$
\end{proof}

\begin{remark}
The above variational principle 	implies that a minimizer $\eta$ to $\LRL$ or $\LRD$ is not necessarily a strict minimizer of $\LL$. However, the difference between the minima will be at most $\varepsilon$. In other words, the system will be $2\varepsilon$-lumpable and -decomposable with respect to $\eta$. Thus, for practical purposes, we can expect
\begin{align}
\label{eq:LRL_optimization_problem}
\xi_L &:= \argmin_{\vartheta\in S(\X,\Z)} \LRL(\vartheta)\\
\shortintertext{and}
\label{eq:LRD_optimization_problem}
\xi_D &:= \argmin_{\vartheta\in S(\X,\Z)} \LRD(\vartheta)
\end{align}
to be ``quasi-optimal'' RCs.
\end{remark}

\section{Numerical approximation of the loss function}
\label{sec:numerical_aspects}

In order to solve the above optimization problems, the loss functions $\LRL$ and $\LRD$, and, depending on the optimization scheme, also their gradients, need to be evaluated numerically for candidate RCs $\vartheta$. The presumed high dimension $n$ of the state space however presents several challenges, which are discussed in the following.

\subsection{Sampling requirements of the decomposability loss function}
\label{sec:montecarlo}
Note that at the core of both $\LRL$ and $\LRD$ lies the evaluation of the transition density $p^t$ at certain start and end points. The function $p^t$ is however not known analytically in practice, and must be estimated empirically by simulations of the dynamics. A critical question is thus how many of these estimates are required in order to approximate $\LRL$ or $\LRD$ up to a given tolerance. In particular, an exponential dependency of that number on the dimension $n$ would be fatal. We will now show that, if the system is indeed highly lumpable with respect to some RC $\xi$ (which in general is different from the candidate RC $\vartheta$), the differential decomposability loss function $\LRD$ can be approximated very efficiently. Crucially, $\xi$ does not have to be known, its implied existence is sufficient to guarantee the efficiency.

By expanding the $\|\cdot\|_{L^1_\mu}$ norm in~\eqref{eq:LVD}, we can write~$\LRD$~as
\begin{align}
\LRD(\vartheta) &= \frac{1}{|\Z|}\int_\Z \int_{\Sigma_\vartheta(z)} \int_{\Sigma_\vartheta(z)} \int_\X \left| \frac{p(x,y^{(1)})}{\pi(y^{(1)})} - \frac{p(x,y^{(2)})}{\pi(y^{(2)})} \right|\pi(x)\ts dx\ts d\mu_z(y^{(1)})\ts d\mu_z(y^{(2)})\ts dz. \label{eq:LVD_expanded}
\intertext{or, by changing the integration order, as}
	&= \frac{1}{|\Z|} \int_\X \left(\int_\Z \int_{\Sigma_\vartheta(z)} \int_{\Sigma_\vartheta(z)} \left| \frac{p(x,y^{(1)})}{\pi(y^{(1)})} - \frac{p(x,y^{(2)})}{\pi(y^{(2)})} \right|\ts d\mu_z(y^{(1)})\ts d\mu_z(y^{(2)})\ts dz\right) \ts \pi(x)\ts dx. \label{eq:LVD_expanded2}
\end{align}

Exact evaluation of the outermost integral in~\eqref{eq:LVD_expanded2} would require analytical knowledge of the transition kernel $p$. Specifically, it would require knowledge of \emph{all} transition densities, i.e., of $p(x,\cdot)$ for \emph{all} starting points $x\in \X$, which we cannot assume in practice. We can however assume that a sufficiently precise approximation of a certain \emph{limited} number of transition densities $p(x^{(k)},\cdot),~k=1,\ldots,M$ can be obtained. This approximation would typically be realized by parallel simulation of the stochastic dynamics starting from $x^{(k)}$, thus creating a point cloud that is distributed according to $p(x^{(k)},\cdot)$. One can then apply density estimation techniques such as kernel- or spectral density estimation to obtain an analytic expression of $p(x^{(k)},\cdot)$.
In any realistic scenario, we would however not to be able to approximate $p(x^{(k)},\cdot)$ for a uniform covering of $\X$ by a sufficient amount of points,
due to the typically high dimension of $\X$. For example, a regular grid-based covering of the unit cube $[0,1]^n$ with spacing $\delta>0$ would require $M=\left(\frac{1}{\delta}\right)^n$ grid points.
This exponential dependence of $M$ on the dimension renders classical numerical quadrature schemes infeasible, leaving Monte Carlo (MC) quadrature as the only viable option for solving the integral over $\X$.

To quantify the MC error in $M$ for this integral, define
\begin{align}
\label{eq:MCintegrand}	f(x) &:=	 \frac{1}{|\Z|}\int_\Z \int_{\Sigma_\vartheta(z)} \int_{\Sigma_\vartheta(z)} \left| \frac{p(x,y^{(1)})}{\pi(y^{(1)})} - \frac{p(x,y^{(2)})}{\pi(y^{(2)})} \right|\ts d\mu_z(y^{(1)})\ts d\mu_z(y^{(2)})\ts dz, \\
\label{eq:MCintegral_exact}	I(f) &:= \int_\X f(x) \ts d\mu(x), \\
\label{eq:MCintegral}	I_M(f) &:= \frac{1}{M} \sum_{i=1}^M f(x^{(k)}),\quad x^{(k)} \sim \mu~i.i.d..
\end{align}
For the expected approximation error then holds~\cite[p.~77]{kalos_monte_2009}
\begin{align}
	\mathbb{E}\left[ \left| I(f) - I_M(f) \right| \right]&= \frac{\var_\mu(f)}{\sqrt{M}}, \label{eq:MCerror}
\shortintertext{where $\var_\mu(f)$ denotes the variance of $f$ with respect to $\mu$, defined by}
	\var_\mu(f) &= \mathbb{E}_\mu\left[f^2\right] - \left(\mathbb{E}_\mu[f]\right)^2. \notag
\end{align}
The independence of the convergence rate $1/\sqrt{M}$ from the dimension is what gives MC methods an edge above conventional methods, at least in theory. However, the \emph{effective} convergence speed, and hence the required number of start points $x^{(k)}$, is highly influenced by the prefactor $\var_\mu(f)$.

We will show now that for highly lumpable systems, $\var_\mu(f)$ tends to be substantially smaller than for non-lumpable systems. The intuitive explanation is that, for systems that are lumpable with respect to some RC $\xi$, $f$ is essentially constant along every level set of~$\xi$, hence only the variation of $f$ along $\xi$ contributes to $\var_\mu(f)$. This holds even in the candidate RC $\vartheta$ that appears in the definition of $f$ is not close to $\xi$. The intuition is formalized by the following theorem:

\begin{theorem}
\label{thm:f_variance}
Assume that the system is $\varepsilon$-lumpable with respect to $\xi:\X\rightarrow \Z$ and the effective density $p_L:\Z\times\X\rightarrow \R$. Define $f_L:\Z\rightarrow \R$ by
\begin{equation}
\label{eq:MCintegrandL}
f_L(z) := \frac{1}{|\Z|}\int_\Z \int_{\Sigma_\vartheta(z')} \int_{\Sigma_\vartheta(z')} \left| \frac{p_L(z,y^{(1)})}{\pi(y^{(1)})} - \frac{p_L(z,y^{(2)})}{\pi(y^{(2)})} \right|\ts d\mu_z(y^{(1)})\ts d\mu_z(y^{(2)})\ts dz'.
\end{equation}
Furthermore, define the \emph{effective invariant density} by
\begin{equation}
	\bar{\pi}(z) := \int_{\Sigma_\xi(z)}\pi(x)\ts d\sigma_z(x)
\end{equation}
where $\sigma_z$ denotes the surface measure on $\Sigma_\xi(z)$. Then there exists a constant $C>0$ such that
$$
\left| \var_\mu(f) - \var_\mu(f_L\circ\xi) \right| \leq 2(1+C^2)\|f_L\circ \xi\|_{L^1_\mu} \|\bar{\pi}\|_\infty \varepsilon + \mathcal{O}(\varepsilon^2).
$$
\end{theorem}

In words, the variance of $f$ is $\varepsilon$-close to the variance of $f_L\circ \xi$.
Also, the variance of $f_L\circ \xi$ is equal to the variance of $f_L$:

\begin{lemma}
\label{lem:fL_variance}
	Let $\bar{\mu}$ be defined as in Lemma~\ref{thm:f_variance} and $h\in L^1_{\bar{\mu}}(\Z)$. Then
	$$
	\var_\mu(h\circ \xi) = \var_{\bar{\mu}}(h).
	$$
\end{lemma}

The proof of Theorem~\ref{thm:f_variance} and Lemma~\ref{lem:fL_variance} can be found in Appendix~\ref{sec:proof_f_variance}.

\markchange{
\begin{remark}
To summarize, the variance of $f$ is $\varepsilon$-close to the variance of $f_L$. As $f_L$ is defined on $\Z$, the variance of $f$ cannot depend on the full phase space dimension $n$ through the dimension of its argument.

Looking at the definition of $f_L$, one sees that the variance of the effective transition density $p_L(z,y)$ in its first argument influences $\operatorname{Var}_{\bar{\mu}}(f_L)$.
Note, however, that the dimension $n$ also indirectly appears in the definition of $f_L$, through the integrations over the $(n-r)$-dimensional level sets $\Sigma_\vartheta(z)$.
As the candidate RC $\vartheta$ can be arbitrarily complex, it is very hard to give any general estimates on the influence of its level sets and $n$ on $\var_{\bar{\mu}}(f_L)$. However, we found no argument why $\var_{\bar{\mu}}(f_L)$ should \emph{increase} with increasing $n$ in general. Indeed, in Section~\ref{sec:example_brownian_motion}, we present an example system for which $\var_{\bar{\mu}}(f_L)$ is inversely correlated to $n$. There, we also numerically confirm the predicted dependence of the Monte Carlo error on $\var_\mu(f)$, and the  predicted dependence of $\var_\mu(f)$ on $\varepsilon$.
\end{remark}
}

\subsection{Further numerical challenges}
\label{sec:further_numerical}

In order to solve the optimization problem~\eqref{eq:LRD_optimization_problem} in practice, several more numerical challenges will need to be overcome. These challenges do however not involve the collection of dynamical data, i.e., the expensive realization of the dynamical system, and hence can be considered peripheral.
We will therefore only briefly sketch possible solutions to these challenges, and leave the exact elaboration to future work.

\subsubsection*{Sampling the invariant measure}
In high dimensions, generating the $\mu$-distributend samples $x^{(k)}$ in~\eqref{eq:MCintegral} by classical techniques such as importance sampling is notoriously inefficient~\cite{au_important_2003}. As an alternative, Markov Chain Monte Carlo techniques~\cite{robert_monte_1999}, such as Metropolis-Hastings or Gibbs Sampling can be used. Due to their Markov chain fulfilling the detailed balance condition w.r.t. $\pi$, their generated samples are precisely $\mu$-distributed, so the variance estimates derived in Section~\ref{sec:montecarlo} hold. The same is true for Boltzmann generators~\cite{noe_boltzmann_2019}, which in addition overcome the problem of sampling long-lived states.

\subsubsection*{Discretization of the solution space}

The solution space $S(\X,\Z)$ for $\xi_D$ in~\eqref{eq:LRL_optimization_problem} needs to be replaced by a finite-dimensional parametric ansatz space. While in principle classical approaches like Galerkin or grid-based finite element discretization methods could be explored, these methods suffer heavily from the curse of dimensionality, that is, the exponential dependence of the number of parameters on the system dimension~\cite{novak_approximation_2009}.

\markchange{
Due to its prevalence in scientific computing, many approaches have been suggested to tackle the curse of dimensionality, including sparse grids~\cite{griebel_sparse_1999}, mesh-free methods~\cite{garg_meshfree_2018}, and non-parametric approaches~\cite{scholkopf_learning_2001}. All of these methods are in principle compatible with the task of evaluating $\LRD(\vartheta)$, provided the level sets $\Sigma_\vartheta(z)$ can be sampled for the discretized $\vartheta$.

However, one particular discretization method practically suggests itself for our overall task of finding the minimizer of $\LRD$, and that is the representation of~$\xi_D$ via a multilayer neural network~\cite{lecun_deep_2015}.} These models have demonstrated impressive performance for problems of high input dimensions, such as image recognition~\cite{krizhevsky_imagenet_2017}, protein structure prediction~\cite{senior_improved_2020} or, quite relevant to our task, Markov model construction via the approximation of transfer operator eigenfunctions~\cite{mardt_vampnets_2018}. Minimization of the loss function $\LRD$ can then be accomplished by gradient descent methods, of which there exist numerous efficient variants and implementations~\cite{ruder_overview_2017}.

\subsubsection*{Level set integrals}

Besides the integral over $\X$, computation of $\LRD(\vartheta)$ requires the numerical solution of an integral over $\Z$ and, for each $z\in\Z$, two surface integrals over the $z$-level set $\Sigma_\vartheta(z)$ of $\vartheta$, see~\eqref{eq:LVD_expanded2}. As the dimension $r$ of $\Z$ was assumed to be small, classical grid-based quadrature methods are sufficient for the former. The level sets $\Sigma_\vartheta(z)$, however, are $n-r$ dimensional nonlinear sub-manifolds of $\X$, and thus too high-dimensional for grid based methods.

The canonical approach would be again to use Monte Carlo quadrature, like for the integral over $\X$. This approach was used for the low-dimensional examples presented in Section~\ref{sec:examples}. Sampling of the surface measure $\mu_z$ can be accomplished by restricted simulation of the dynamics, for which numerical schemes were suggested in~\cite{ciccotti_projection_2007}. Moreover, if $\vartheta$ is modeled by a neural network, sampling $\Sigma_\vartheta(z)$ could be alternated with the optimization step, leading to an optimization scheme similar to stochastic gradient descent.

\section{Numerical examples}
\label{sec:examples}

\subsection{Timescale-separated Brownian motion}
\label{sec:example_brownian_motion}

We consider a process $(X_t) = (X_t^{(1)}, \dots, X_t^{(n)})$ on the $n$-dimensional
torus $\mathbb{T}^n := [-\pi,\pi]^n$ for $n \geq 2$
with slow diffusion in the first coordinate direction (standard Brownian motion)
and instantaneous and pairwise independent equilibrations to the uniform distribution
in the $n-1$ remaining coordinate components. We can write this system in SDE form as
\begin{gather*}
	\mathrm{d} X_t^{(1)} = \mathrm{d}W_t, \\
	X_t^{(i)} \sim U([-\pi, \pi]),  \quad  2 \leq i \leq n,
\end{gather*}
such that
\begin{gather*}
	X_{t_1}^{(i)} \perp X_{t_2}^{(i)}, \quad 1 \leq i \leq n,  t_1 \neq t_2, \\
	X_{t_1}^{(i)} \perp X_{t_2}^{(j)} \quad 1 \leq i \neq  j \leq n, \forall t_1, t_2,
\end{gather*} where $W_t$ denotes a standard Brownian motion, $U$ denotes the uniform distribution, and $\perp$ marks independence.
We introduce a second process
$(\widetilde{X}_t)$ on $\mathbb{T}^n$ which is
parametrized by a variance $\sigma^2 >0$ by the global diffusion
\begin{align*}
	\mathrm{d} \widetilde{X}_t^{(1)} &= \mathrm{d} \widetilde{W}_t^{(1)} \\
	\mathrm{d} \widetilde{X}_t^{(i)} &= \sigma \mathrm{d} \widetilde{W}_t^{(i)}, \quad 2 \leq i \leq n,
\end{align*}
where the $\widetilde{W}_t^{(i)}, 1 \leq i \leq n$ are pairwise independent standard
Brownian motions.

\subsubsection{Transition densities}
The transition densities of $(X_t)$ and $(\widetilde{X}_t)$
can be conveniently described in terms of the one-dimensional
\emph{wrapped normal distribution} \cite{jammalamadaka_topics_2001}
on the circle $[-\pi, \pi]$ with
variance $\sigma^2 \geq 0$, which is
defined by the density $g^\sigma: [-\pi, \pi] \times [-\pi, \pi] \rightarrow \R_+$
given by
\begin{equation*}
	g^\sigma(x,y) :=
	\frac{1}{2 \pi}
	\left( 1 + 2 \sum_{k=1}^\infty \rho^{k^2}
	\cos( k (y-x) ) \right), \textnormal{ where } \rho = \exp(- \sigma^2/2).
\end{equation*}
In particular, for $ x=(x_1, \dots, x_n),  y=(y_1, \dots, y_n) \in \mathbb{T}^n$
we straightforwardly obtain the transition density of $(X_t)$ as
\begin{equation*}
	p^{\tau, \infty}(x,y) := \frac{1}{(2\pi)^{n-1}} \; g^\tau(x_1, y_1)
\end{equation*} as well as the transition density of $(\widetilde{X}_t)$ as
\begin{equation}
\label{eq:transition_density_sigma}
	p^{\tau, \sigma}(x,y) := g^\tau(x^{(1)},y^{(1)}) \, \prod_{i = 2}^{n} g^{\tau\sigma}(x_i, y_i).
\end{equation}
The transition densities are illustrated in Figure~\ref{fig:transitiondensities} for
the case $n=2$.

\begin{figure*}[h]
	\centering
	\includegraphics[scale=1]{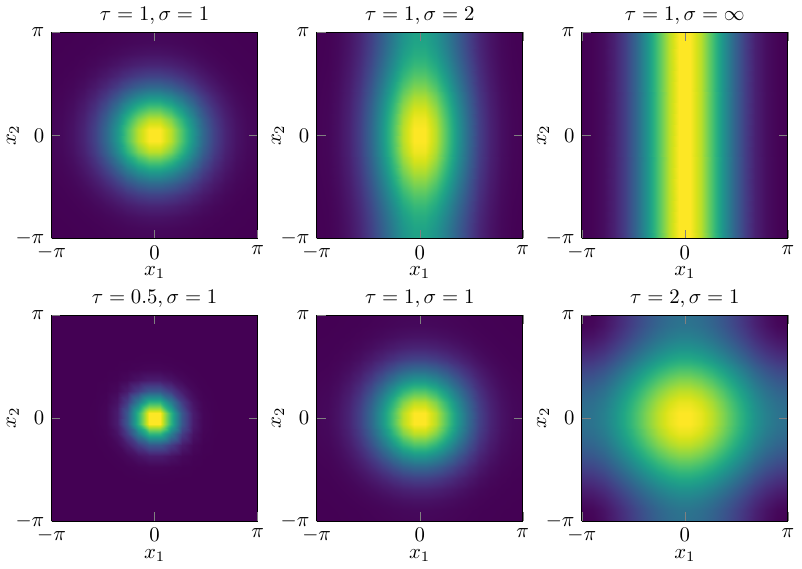}
	\caption{Illustration of transition densities $p^{\tau, \sigma}(0,\cdot)$
		for different lag times $\tau$ and standard deviations $\sigma$ in the case $n=2$.
		The top row illustrates how the transition densities
		$p^{\tau, \sigma}(x, \cdot)$ approximate $p^{\tau, \infty}(x, \cdot)$
		for an increasing standard deviation $\sigma$ when $\tau$ is fixed. Note that
		for fixed $x_1$ coordinates, the limiting density
		$p^{\tau, \infty}(x, \cdot)$ is constant along the $x_2$-coordinate.
		The bottom row illustrates the larger degree of global dispersion for inceasing lag times
		when $\sigma$ is fixed.}
	\label{fig:transitiondensities}
\end{figure*}

\subsubsection{Lumpability}
\label{sec:slowfast_lumpability}
As $\sigma$ increases, the transition density $p^{\tau, \sigma}(x,y) $ uniformly
approximates $p^{\tau, \infty}(x,y) $, we may therefore
understand $(X_t)$ as the limiting process of $(\widetilde{X}_t)$ for large variances.
Because the transition density of $p^{\tau, \infty}(x,y) $ only depends on the first coordinate components $x_1, y_1$,
we can verify lumpability of $(\widetilde{X}_t)$ for large enough $\sigma$, as we will briefly illustrate now.

Consider the RC
$\xi: [- \pi, \pi]^n \rightarrow [- \pi, \pi]$
given by the orthogonal
projection onto the first coordinate component
$\xi((x_1,x_2, \dots, x_n) = x_1$. We call $\xi$ the \emph{optimal} RC. With the effective transition density
\begin{equation*}
	p_L^\tau(z, y) :=  \frac{1}{(2\pi)^{n-1}} \; g^\tau(z, y_1),
\end{equation*}
one can then show that
\begin{equation}
		\label{eq:example_lumpability}
		\big\| p^{\tau, \sigma}(\ast, \cdot)
		- p_L^\tau( \xi(\ast), \cdot)  \big\|_\mathbb{K}
	 = \mathcal{O}
	 \left( \exp(- (\tau \sigma)^2/2)  \right)
\end{equation}
for both $\sigma\to \infty$ or $\tau\to \infty$. The derivation of~\eqref{eq:example_lumpability} can be found in Appendix \ref{sec:lumpability_brownian_motion}.
In particular, for any lag time $\tau>0$, we will find a $\sigma$ for which the system becomes arbitrarily lumpable, while in the limit $\sigma\rightarrow \infty$, the system is ``perfectly'' lumpable (i.e., $\big\| p^{\tau, \sigma}(\ast, \cdot)
		- p_L^\tau( \xi(\ast), \cdot)  \big\|_\mathbb{K} = 0$) for every $\tau>0$.
Note that $\varepsilon$-lumpability of $(\tilde{X}_t)$ also implies $\varepsilon$-decomposability due to Proposition~\ref{prop:equivalence_reversible_epsilon}.

As our estimates of the distance ~$\big\| p^{\tau, \sigma}(\ast, \cdot)
		- p_L^\tau( \xi(\ast), \cdot)  \big\|_\mathbb{K}$ in Appendix \ref{sec:lumpability_brownian_motion} are asymptotic in nature, we can give no precise formula for the dependence of the minimal $\varepsilon$ in~\eqref{eq:lumpability} on $\sigma$. \markchange{In particular, we cannot predict the dependence of $\varepsilon$ on the system dimension $n$. Nevertheless, when computing the distance numerically, its dependence on $n$ appears to be moderate, and seems to diminish with growing $n$ (Figure~\ref{fig:lumpability_error}). Hence, we can expect a high degree of lumpability for moderately large $\sigma$ and moderately high dimensions.}
		Also, the computations confirm the convergence rate predicted in~\eqref{eq:example_lumpability}.

\begin{figure*}[h]
	\centering
	\includegraphics[scale=1]{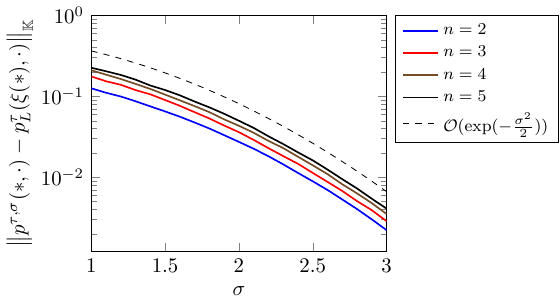}
	\caption{$\mathbb{K}$-distance of the transition density $p^{\tau,\sigma}$ to the effective transition density $p_L^\tau$, indicating the degree of lumpability. We observe the predicted decay in the diffusion constant $\sigma$, as well as a relative insensitivity to the system dimension $n$.}
	\label{fig:lumpability_error}
\end{figure*}

\subsubsection{Loss function illustration}

Next, we investigate the dependence of the loss function $\LRD$ on the RC, for the two-dimensional process (one slow and one fast component). To be precise, we investigate the dependence of $\LRD(\vartheta_\alpha)$ on the parameter $\alpha$ of the family of ``test'' reaction coordiantes
\begin{equation}
\label{eq:xialpha}
\vartheta_\alpha(x) := \frac{1}{\pi (\cos|\alpha| +\sin |\alpha|)} \begin{pmatrix}
	\cos \alpha & \sin \alpha
\end{pmatrix}
\cdot
\begin{pmatrix}
	x_1 \\ x_2
\end{pmatrix},\quad \alpha\in(-\pi,\pi).
\end{equation}
for different values of the diffusion constant $\sigma$. The level sets of $\vartheta_\alpha$ are one-dimensional hyperplanes that intersect the $x_1$-axis with angle $\pi/2-\alpha$. The prefactor in~\eqref{eq:xialpha} ensures that $\mathbb{Z}=\operatorname{range}(\vartheta_\alpha)=[-1,1]$ for all~$\alpha$. Figure~\ref{fig:ansatzRC} illustrates $\vartheta_\alpha$ and its level sets.
\begin{figure*}[h]
	\centering
	\includegraphics[scale=1]{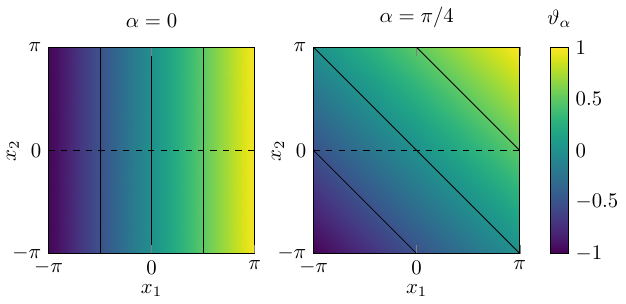}
	\caption{The linear RC $\vartheta_\alpha$ for two different values of $\alpha$.}
	\label{fig:ansatzRC}
\end{figure*}

We use a combination of symbolic computation and numerical quadrature techniques for computing $\LRD(\vartheta_\alpha)$, using Mathematica and Matlab.
In particular, for sampling points on the level sets~$\Sigma_{\vartheta_\alpha}(z)$, required for numerically computing the two innermost integrals in~\eqref{eq:LVD_expanded2}, we utilized the Matlab function \verb+randFixedLinearCombination+, written by John D'Errico and provided through MatlabCentral~\cite{derrico_randfixedlinearcombination_2021}. The scripts containing our computations are provided in the SI.
We do however \emph{not} suggest to use these scripts for the computation and minimization of $\LRD$ in real-world systems, as in particular the integration over the $(n-r)$-dimensional level sets quickly become infeasible. See the discussion in Section~\ref{sec:further_numerical} for a first outlook on how we plan to solve the optimization problem in practice.

The differential decomposability loss function $\LRD(\vartheta_\alpha)$ in dependence of $\alpha$ is shown in Figure~\ref{fig:lossfun}. We observe that, for $\sigma=2$ and $\sigma=\infty$, $\LRD(\vartheta_\alpha)$ indeed takes its unique global minimum for $\alpha=0$, i.e., the optimal RC $\vartheta_0(x) = \xi(x) = x_1$. We also observe that for $\sigma=2$, the minimum value is not zero, as the system is not ``perfectly lumpable'' with respect to $\vartheta_0$, whereas for $\sigma=\infty$ it is. Further, the ``worst'' RCs $\vartheta_{\pm\pi/2}(x)= \pm x_2$, which project onto the system's fast instead of its slow coordinate, correctly get assigned the global maximum value.
For $\alpha=1$, on the other hand, the diffusion is isotropic, and hence the ``slow directions'' coincide with the directions of largest extent, which are the diagonals of the domain $[-\pi,\pi]^2$. Consequently, $\LRD(\vartheta_\alpha)$ becomes minimal for $\alpha=\pm \pi/4$, for which $\vartheta_\alpha$ projects onto the diagonals $\{x_1 =\pm x_2\}$.

 The behavior of $\LRD$ is therefore perfectly consistent with our intuition, and we can expect to identify the optimal RC by minimizing $\LRD$.

\begin{figure*}[t]
	\centering
	\includegraphics[scale=1]{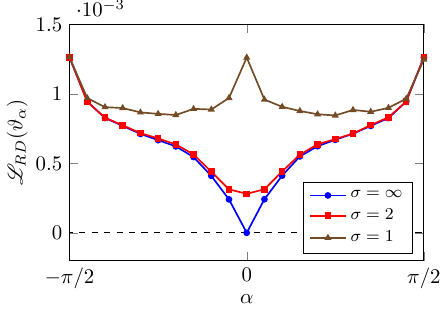}
	\caption{The differential decomposability loss function of the RC $\vartheta_\alpha$ for $\alpha\in(-\pi/2,\pi/2)$ and different values of $\sigma$.}
	\label{fig:lossfun}
\end{figure*}

\subsubsection{Loss function Monte Carlo error}

Finally we investigate the Monte Carlo quadrature error for the integral over $\X$ in $\LRD$.
As discussed in Section~\ref{sec:montecarlo}, this error is, besides the number of sample points $M$, determined by the variance of the integrand $f$, which for the current system takes the form
$$
f(x) =	 \frac{(2\pi)^n}{|\vartheta_\text{max}-\vartheta_\text{min}|}\int_{\vartheta_\text{min}}^{\vartheta_\text{max}} \int_{\Sigma_\vartheta(z)} \int_{\Sigma_\vartheta(z)} \left| p^{\tau,\sigma}(x,y^{(1)}) - p^{\tau,\sigma}(x,y^{(2)})\right|\ts d\sigma_z(y^{(1)})\ts d\sigma_z(y^{(2)})\ts dz,
$$
where $\vartheta_\text{max}$ and $\vartheta_\text{min}$ describe the maximum and minimum of $\vartheta$.
We point out in particular that $f$ depends on the diffusion coefficient $\sigma$ and, through the RC $\vartheta$, on the system dimension $n$. The influence of these two parameters on the Monte Carlo error is the primary subject of investigation for this section.

We consider in dimension $2$ and $3$, respectively, the test RCs
$$
\vartheta^2(x) = x_1+x_2,\qquad \vartheta^3(x) = x_1+x_2+x_3.
$$
Figure~\ref{fig:integrand_sigma} shows $f$ for the RC $\vartheta^2$ and different values of $\sigma$. We observe that with increasing $\sigma$, $f$ indeed becomes increasingly constant on the level sets $\{x_1 = z\}$ of the optimal RC $\xi$. This behavior was predicted in Section~\ref{sec:montecarlo}.

\begin{figure*}[t]
	\centering
	\includegraphics[scale=1]{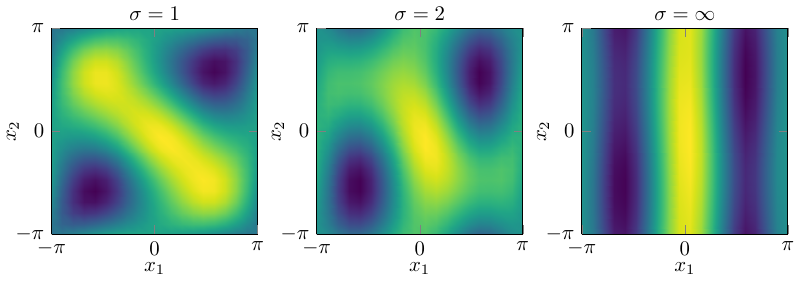}
	\caption{The function $f$, i.e., the integrand of the integral over $\X$ in $\LRD(\vartheta)$ for $\vartheta(x)=x_1+x_2$ and different values of $\sigma$.}
	\label{fig:integrand_sigma}
\end{figure*}

Figure~\ref{fig:variance_sigma} illustrates the variance of $f$. We first note that the variance indeed decreases for increasing $\sigma$, and converges towards the variance of the process with instantaneously equilibrating components $x_2,\ldots,x_n$ (equivalent to choosing ``$\sigma=\infty$''). Moreover, we observe that the variance of the three-dimensional process is substantially smaller compared to the two-dimensional process. This observation still holds when considering the relative variance $\var[f]/\mathbb{E}[f]$. This demonstrates that higher-dimensional function do not per se possess higher variance. In fact, $\var[f]$ appears to be more dependent on the choice of the particular RC $\vartheta$ than on $n$, although, to avoid digression, we will refrain from detailed analysis.

\begin{figure*}[t]
	\centering
	\includegraphics[scale=1]{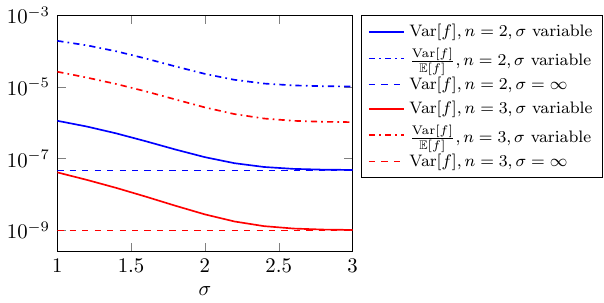}
	\caption{(Relative) variance of $f$ associated in dimensions 2 and 3 for various $\sigma$. We observe convergence of $\var[f]$ with variable $\sigma$ towards $\var[f]$ associated with the $\sigma$-independent limit process $X_t$.}
	\label{fig:variance_sigma}
\end{figure*}

Finally, we estimate the relative expected MC error
$$
\frac{\mathbb{E}\left[ \left| I(f) - I_M(f)\right| \right]}{I(f)}
$$
where $I(f)$ and $I_M(f)$ are the exact integral and Monte Carlo integral with $M$ samples of $f$ defined in~\eqref{eq:MCintegral_exact} and~\eqref{eq:MCintegral}.
In practice, this error indicates how many ``dynamical samples'' $p^{\tau,\sigma}(x^{(k)},\cdot)$ need to be created by numerical simulation in order to approximate $\LRD(\vartheta)$ up to a given accuracy. In the present example, however, $p^{\tau,\sigma}$ is known analytically by~\eqref{eq:transition_density_sigma}.

Figure~\ref{fig:MCerror} shows the relative error in dependence on the sample size $M$ for $\vartheta^2$ and~$\vartheta^2$, each for a finite value of $\sigma$ as well as $\sigma=\infty$.
In all four cases we observe the expected Monte Carlo convergence rate of $\mathcal{O}(1/\sqrt{M})$.
Moreover, the error is significantly smaller for higher $\sigma$, i.e., the more lumpable system.
Finally, the expected error for the higher dimensional system is slightly smaller.
All these phenomena are perfectly consistent with the preceding analysis of~$\var[f]$.

\begin{figure*}[t]
	\centering
	\includegraphics[scale=1]{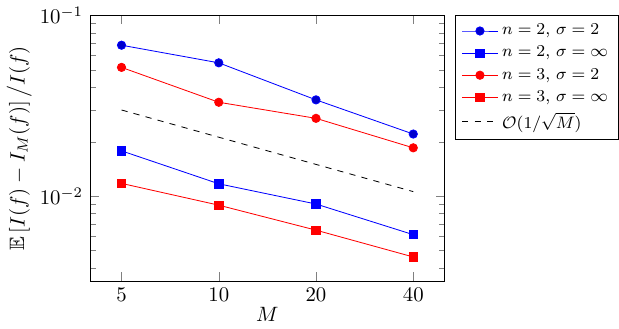}
	\caption{The relative Monte Carlo quadrature error for the integral over $\X$ in $\LRD(\vartheta_w)$ for different dimensions $n$ and different diffusion coefficients $\sigma$.}
	\label{fig:MCerror}
\end{figure*}

We conclude from this example that for lumpable systems, we can expect to find the optimal RC by minimising $\LRD$, that the Monte Carlo approximation to $\LRD(\vartheta)$ requires only few dynamical samples for highly lumpable systems, and that a high dimension of the base system has no negative impact on the performance.

\subsection{Metastable circular system}
\label{sec:example_metastable_circular}

To demonstrate the behavior of the loss function for nonlinear optimal RCs, we consider as a second example a two-dimensional system governed by the overdamped Langevin dynamics
\begin{equation}
	\label{eq:overdamped_langevin}
	dX_t = -\nabla V(X_t) + \sqrt{2\beta^{-1}} dW_t,
\end{equation}
where $V:\mathbb{R}^2\rightarrow \mathbb{R}$ is the potential energy surface and $\beta >0$ is the inverse temperature determining the strength of the Brownian motion $W_t$.
Informally, movement of this system can be described as a random walk within the energy landscape, aiming in the direction of steepest descent of $V$ but being
disturbed by temperature-scaled white noise.

In particular, we here consider the family of potentials
$$
V_\sigma(x) = \cos\left(5 \varphi(x)\right) + \sigma \left(r(x)-1\right)^2,\quad \sigma>0,
$$
where $\left(\varphi(x), r(x)\right)$ describe the polar coordinates of $x$, i.e.,
$$
\varphi(x)=\operatorname{atan2}(x),\qquad r(x)=\|x\|_2.
$$
The potential consists of two components: a cosine term in the angular coordinate with five local minima of equal depth, and a quadratic term in the radial coordinate with a single minimum at $r=1$.
The full potential, shown in Figure~\ref{fig:lemonslice_potential}, therefore possesses five local minima arranged along the unit circle. The two components $\varphi$ and $r$ as functions on $\X$ are shown in Figure~\ref{fig:lemonslice_lossfun}~(left).

\begin{figure*}
	\centering
	\includegraphics[scale=1]{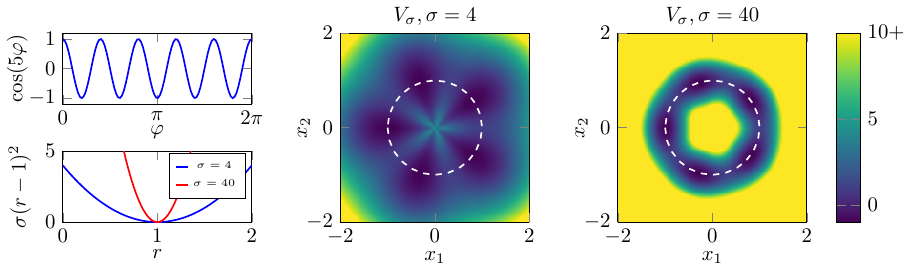}
	\caption{Left: The two components of the potential, which depend on the angular and the radial part of the polar coordinates, respectively. Center \& right: The circular potential and its components for different values of $\sigma$. The dashed white line indicates the unit circle, i.e., the minimal energy pathway connecting the local minima.}
	\label{fig:lemonslice_potential}
\end{figure*}

\subsubsection{Metastability analysis}

At moderate temperatures, the five local minima of $V$ induce \emph{metastability} in the system. This means that a typical trajectory will vibrate around a minimum for a long time, until suddenly, induced by sufficiently strong stochastic excitation, undergo a rapid transition to another local minimum.
If additionally the equilibration in the radial direction is sufficiently fast, i.e., the parameter $\sigma$ is sufficiently large, then these rare transitions will be the slowest sub-processes of the system. To confirm this, we compute the leading eigenvalues and associated eigenfunctions of the system's transfer operator $\mathcal{P}^t:L^2(\X)\rightarrow L^2(\X)$, which is defined by
$$
\mathcal{P}^tf(x) = \int_\X f(y) p^t(y,x)\ts dy.
$$
Going back to the late 90s, spectral analysis of the transfer operator, or its adjoint, the Koopman operator, forms the basis of many model reduction methods that aim to preserve the system's long timescales, for an overview see~\cite{klus_data-driven_2018}.

The eigenvalues of $\mathcal{P}^t$ for $t=0.1$ and various values of $\sigma$ are shown in Figure~\ref{fig:lemonslice_eigenvalues}. We see that for $\sigma=10$ and $\sigma=100$, there is a significant \emph{spectral gap} after $\lambda_5$, indicating a significant time scale separation between the associated sub-processes. Analysing the sign structure of the associated eigenfunctions  confirms that the slowest processes are indeed associated with the transitions between the metastable sets (see the SI).

For small $\sigma$, on the other hand, no spectral gap can be observed, hence the metastable transitions are not considerably slower than the remaining processes. Indeed, sign analysis of the sixth eigenfunction for $\sigma=1$ shows that the corresponding sub-process describes the equilibration in the radial direction (see the SI). As $\lambda_6$ is not significantly separated from $\lambda_5$, the radial equilibration occurs on roughly the same timescale as the metastable transitions.

\begin{figure*}
	\centering
	\includegraphics[scale=1]{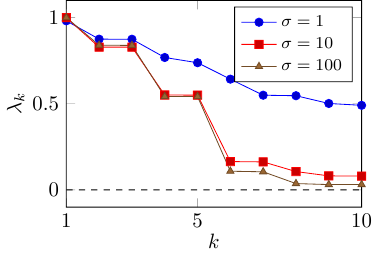}
	\caption{Leading eigenvalues of the circular system's transfer operator for different values of $\sigma$. For larger $\sigma$, we observe a spectral gap, which indicates that the equilibration times of the metastable transitions are much slower than that of other sub-processes in the system.}
	\label{fig:lemonslice_eigenvalues}
\end{figure*}

\subsubsection{Lossfunction computation}

The preceding analysis confirms that, for high enough $\sigma$, a good RC should resolve the transitions between the metastable sets. As these transitions occur predominantly within certain ``transition channels'' that surround the minimum energy pathways (MEPs)~\cite{metzner_illustration_2006}, and as in our system the MEPs between two neighbouring minima are segments of the unit circle, we would expect the angular component $\varphi(x)$ of the polar coordinates of $x$ to be a good RC. Conversely, we would expect the radial component $r(x)$ to be a bad RC, especially for high values of $\sigma$.

To test this hypothesis, we now compute the differential lumpability loss function $\LRD$ for both~$\varphi$ and~$r$.
For the lag time parameter in $\LRD$, we use $\tau=0.1$, as the observed spectral gap for this time indicates that the time scale separation between the slow and fast processes has already manifested itself.
We again solve the various integrals in~\eqref{eq:LVD_expanded2} by Monte Carlo quadrature.
One complication over the simple slow-fast system is that the transition density functions $p^\tau(x,\cdot)$ are not known analytically, so they have to be approximated empirically. For that, we first draw an empirical sample of $p^t(x,\cdot)$ by simulating many trajectories with starting point $x$, and then apply the kernel density estimation algorithm to the end points.
For details on the numerical implementation see the SI.

Figure~\ref{fig:lemonslice_lossfun}~(right) shows the loss function for the two RCs in dependendce of~$\sigma$. We see that, indeed, $\LRD(\varphi) < \LRD(r)$ for all values of $\sigma$, hence $\varphi$ is the better RC. Moreover, for increasing $\sigma$, $\LRD(\varphi)$ continually decreases\footnote{Note that the flattening of $\LRD(\varphi)$ towards the right side of the plot is a numerical artefact. For high degrees of lumpability, miniscule differences between transition densities need to be quantified, which presents a challenge to our fixed-bandwidth kernel density estimator.}, whereas $\LRD(r)$ increases. This agrees well with our intuitive understanding of the role of $\sigma$: with increasing $\sigma$, the radial component $r$ equilibrates more quickly, so the long-term future evolution of $X_t$ depends more and more only on $\varphi(X_t)$ (i.e., the degree of lumpability with respect to $\varphi$ increases), and less and less on $r(X_t)$ (i.e., the degree of lumpability with respect to $r$ decreases).

\begin{figure*}
	\centering
	\includegraphics[scale=1]{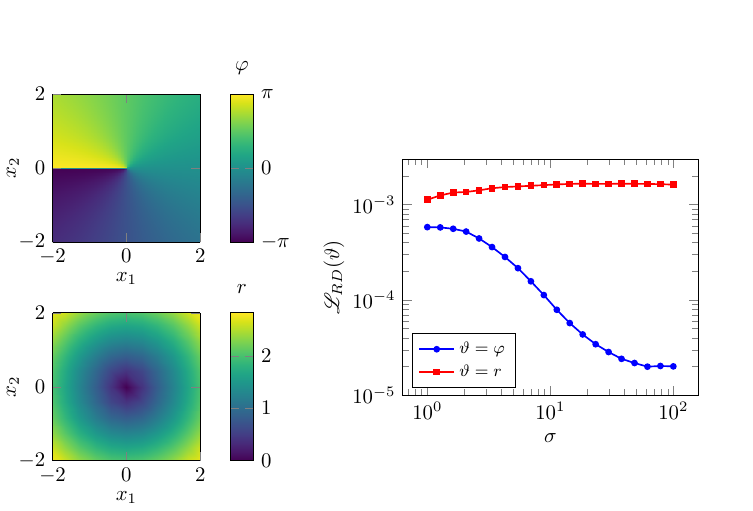}
	\caption{Left: the two components $\varphi$ and $r$ as functions on $\X$. Right: Loss function $\LRD$ for $\varphi$ and $r$ interpreted as RCs.}
	\label{fig:lemonslice_lossfun}
\end{figure*}

\section{Discussion and outlook}
\label{sec:conclusions}

In this paper, we formally characterized optimal reaction coordinates (RCs) in continuous Markovian systems, that is, observables that optimally describe latent long-term mechanisms of the dynamics. We have seen, by both theoretical analysis (Section~\ref{sec:examples_reducible}) and numerical examples (Section~\ref{sec:examples}), that our definition is applicable to several common types of multiscale systems, such as slow-fast systems and metastable systems.
To add further interpretability to our definition, it would be desireable to draw a formal connection to the well-established transition path theory~\cite{e_towards_2006, metzner_illustration_2006}, which characterizes RCs in terms of committor functions and minimum energy pathways on the potential energy surface. In~\cite{bittracher_data-driven_2018}, a connection between the transition manifold approach (a predecessor to the present characterization) and transition path theory has been discussed rather informally, but a rigorous investigation is still outstanding.

We then presented a variational formulation of this characterization as a computational strategy for the discovery of optimal RCs. In particular, that variational formulation provides a leverage point for modern machine learning methods, such as deep learning and stochastic gradient descent.
The implementation of these methods and their demonstration of efficiency are subject of ongoing work. The reduced data requirement for the optimization problem that was demonstrated in this paper raises confidence that we will be able to apply our method to high dimensional problems such as the identification of RCs in large molecular systems.

\section*{Acknowledgements}

This research has been funded by Deutsche Forschungsgemeinschaft (DFG) through grant CRC~1114 ``Scaling Cascades in Complex Systems'', Project Number 235221301, Project B03 ``Multilevel coarse graining of multiscale problems'' and Project A01 ``Coupling a multiscale stochastic precipitation model to large scale atmospheric flow dynamics'', and by Deutsche Forschungsgemeinschaft (DFG) through grant EXC~2046 ``MATH+'', Project Number 390685689, Project AA1-2 ``Learning Transition Manifolds and Effective Dynamics of Biomolecules''.

\bibliographystyle{abbrv}
\bibliography{ContinuousProbabilisticAggregation.bib}

\begin{appendix}

\section{Extended multiscale expansion of slow-fast systems}
\label{sec:extended_multiscale}

In this section, we continue the multiscale analysis for the slow-fast system from Section~\ref{sec:slowfast}. That is, we derive an evolution equation for the dominant component $q^t_0$ of the transition density. Comparing the terms of order $\varepsilon^0$ in~\eqref{eq:multiscale_equation_Smolu} yields
\begin{equation}
\label{eq:order_eps0_Smolu}
\partial_t q_0^t(x,\cdot) = \mathcal{L}_y q_1^t(x,\cdot) + \mathcal{L}_z q_0^t(x,\cdot).
\end{equation}
The middle term is zero, which can be seen as follows: Let the averaging operator $\Pi:L^1_\mu(\X)\rightarrow L^1_{\bar{\mu}}(\Z)$ be defined by
\begin{equation}
\label{eq:averaging_operator}
\Pi g (z) = \frac{\int_\Y g(y,z) \pi(y,z)\ts dy}{\bar{\pi}(z)} \quad\text{where}\quad \bar{\pi}(z) = \int_\Y \pi(y,z)\ts dy
\end{equation}
and $\bar{\mu}$ is the measure induced by $\bar{\pi}$.
As $q_0^t(x,\cdot)$ is independent of $y$ in its second argument, we have
$$
\partial_t q_0^t(x,\cdot) = \frac{1}{\bar{\pi}(z)} \Pi \partial_t q_0^t(x,\cdot) \quad \text{and}\quad \mathcal{L}_z q_0^t(x,\cdot)= \frac{1}{\bar{\pi}(z)} \Pi \mathcal{L}_z q_0^t(x,\cdot).
$$
Furthermore, $\Pi \mathcal{L}_y = 0$. Hence, applying $\frac{1}{\pi(z)}\Pi$ to~\eqref{eq:order_eps0_Smolu} leaves the first and third term invariant, while eliminating the second term, from which follows that $\mathcal{L}_y q_1^t(x,\cdot) = 0$.

Therefore, applying $\Pi$ to~\eqref{eq:order_eps0_Smolu} gives
$$
\partial_t\Pi q_0^t(x,\cdot) = \Pi \mathcal{L}_z q_0^t(x,\cdot).
$$
As both $q_0^t(x,\cdot)$ and $\mathcal{L}_y g$ (for all $g$) are independent of $y$, we finally get for the evolution equation for $q_0^t$
$$
\partial_tq_0^t(x,\cdot) =\mathcal{L}_y q_0^t.
$$
Hence, with $\overline{\mathcal{L}} := \Pi \mathcal{L}_z \Pi$, the evolution equation of $q_0^t(x,\cdot)$ up to order $\varepsilon$ becomes
$$
\partial_t q_0^t(x,\cdot) = \overline{\mathcal{L}} q_0^t(x,\cdot) + \mathcal{O}(\varepsilon)\quad \text{for all } x\in\X.
$$

\section{Proof of Theorem~\ref{thm:f_variance}}
\label{sec:proof_f_variance}

The proof of
consists of simple integral and norm estimations. The main argument, used multiple times in the final proof, is formulated in the following lemma:

\begin{lemma}
	\label{lem:f_L1_estimation}
Assume that the system is $\varepsilon$-lumpable with respect to $\xi:\X\rightarrow \Z$ and the effective density $p_L:\Z\times\X\rightarrow \R$.
Let $f_L$ and $\var{\pi}$ be defined as in Theorem~\ref{thm:f_variance}. Then
\begin{equation}
	\label{eq:f_L1_estimation}
	\left\|f - f_L\circ \xi\right\|_{L^1_\mu} \leq \left\| \bar{\pi}\right\|_\infty \varepsilon.
\end{equation}
\end{lemma}
\begin{proof}
Writing out the left hand side of~\eqref{eq:f_L1_estimation}, we get
	\begin{align*}
\left\|f - f_L\circ \xi\right\|_{L^1_\mu}  &= \int_\X \Bigg|  \left( \frac{1}{|\Z|}\int_\Z \int_{\Sigma_\vartheta(z)} \int_{\Sigma_\vartheta(z)} \Big| \frac{p(x,y^{(1)})}{\pi(y^{(1)})} - \frac{p(x,y^{(2)})}{\pi(y^{(2)})}\Big|\ts d\mu_z(y^{(1)})\ts d\mu_z(y^{(2)})\ts dz \right)\\
&\qquad - \left( \frac{1}{|\Z|}\int_\Z \int_{\Sigma_\vartheta(z)} \int_{\Sigma_\vartheta(z)} \Big| \frac{p_L(\xi(x),y^{(1)})}{\pi(y^{(1)})} - \frac{p_L(\xi(x),y^{(2)})}{\pi(y^{(2)})}\Big|\ts d\mu_z(y^{(1)})\ts d\mu_z(y^{(2)})\ts dz\right) \Bigg|\ts d\mu(x).
\end{align*}
Applying the reverse triangle inequality, this becomes
\begin{align*}
\left\|f - f_L\circ \xi\right\|_{L^1_\mu} &\leq  \int_\X \frac{1}{|\Z|}\int_\Z \int_{\Sigma_\vartheta(z)} \int_{\Sigma_\vartheta(z)} \bigg|\frac{p(x,y^{(1)})}{\pi(y^{(1)})} - \frac{p(x,y^{(2)})}{\pi(y^{(2)})} \\
&\qquad - \frac{p_L(\xi(x),y^{(1)})}{\pi(y^{(1)})} + \frac{p_L(\xi(x),y^{(2)})}{\pi(y^{(2)})}\bigg|\ts d\mu_z(y^{(1)})\ts d\mu_z(y^{(2)})\ts \ts d\mu(x)\ts dz\\
&\leq \int_\X \frac{1}{|\Z|}\int_\Z \int_{\Sigma_\vartheta(z)} \int_{\Sigma_\vartheta(z)} \bigg|\frac{p(x,y^{(1)})}{\pi(y^{(1)})} - \frac{p_L(\xi(x),y^{(1)})}{\pi(y^{(1)})} \bigg| \ts d\mu_z(y^{(1)})\ts d\mu_z(y^{(2)})\ts dz \ts d\mu(x)\\
&\qquad + \int_\X \frac{1}{|\Z|}\int_\Z \int_{\Sigma_\vartheta(z)} \int_{\Sigma_\vartheta(z)} \bigg|\frac{p(x,y^{(2)})}{\pi(y^{(2)})} - \frac{p_L(\xi(x),y^{(2)})}{\pi(y^{(2)})} \bigg| \ts d\mu_z(y^{(1)})\ts d\mu_z(y^{(2)})\ts dz \ts d\mu(x).
\end{align*}
In each of the two summands, the ingrand is independent of $y^{(2)}$ and $y^{(1)}$, respectively, hence one integral over $\Sigma_\vartheta(z)$ becomes the factor $\bar{\pi}(z)$:
\begin{align*}
\left\|f - f_L\circ \xi\right\|_{L^1_\mu} &\leq 2\ts\int_\X \frac{1}{|\Z|}\int_\Z \int_{\Sigma_\vartheta(z)} \left| p(x,y) - p_L(\xi(x),y)\right|\bar{\pi}(z)\ts d\sigma_z(y)\ts dz\ts  d\mu(x)\\
&\leq 2\ts  \left\| \bar{\pi}\right\|_\infty \int_\X \frac{1}{|\Z|}\int_\Z \int_{\Sigma_\vartheta(z)} \left| p(x,y) - p_L(\xi(x),y)\right|\ts d\sigma_z(y)\ts dz\ts  d\mu(x)
\end{align*}
By the coarea formula, the inner two integrals simply describe the integration over $\X$ with respect to the Lebesgue measure, i.e.,
\begin{align*}
\left\|f - f_L\circ \xi\right\|_{L^1_\mu} &\leq 2 \left\| \bar{\pi}\right\|_\infty \frac{1}{|\Z|} \int_\X  \left\| p(x,\cdot) - p_L(\xi(x),\cdot)\right\|_{L^1} \ts d\mu(x).
\intertext{The integral over $\X$ is the $L^1_\mu$ norm, so the overall expression is the $\mathbb{K}$-norm (see~\eqref{eq:K_norm}):}
\left\|f - f_L\circ \xi\right\|_{L^1_\mu} & \leq 2 \left\| \bar{\pi}\right\|_\infty \frac{1}{|\mathbb{Z}|} \left\| p(\ast,\cdot) - p_L(\xi(\ast),\cdot) \right\|_\mathbb{K} \\
\shortintertext{which by the $\varepsilon$-lumpability assumption \eqref{eq:lumpability} is}
&\leq 2\ts \left\| \bar{\pi}\right\|_\infty \varepsilon.
\end{align*}

\end{proof}

The proof of the main result now consists only of reducing the difference of the variances to  the expression $\|f - f_L\circ \xi\|_{L^1_\mu}$.

\begin{proof}[Proof of Theorem~\ref{thm:f_variance}]
We have
\begin{align*}
\left| \var_\mu(f) - \var_\mu(f_L\circ\xi) \right| &= \left| \mathbb{E}_\mu\left[f^2 - (f_L\circ\xi)^2\right] - \left( \mathbb{E}_\mu[f]^2-\mathbb{E}_\mu[f_L\circ\xi]^2\right) \right|\\
	&\leq \underbrace{\left| \mathbb{E}_\mu\left[f^2 - (f_L\circ\xi)^2\right] \right|}_{=:(\star)} + \underbrace{\left| \mathbb{E}_\mu[f]^2-\mathbb{E}_\mu[f_L\circ\xi]^2\right|}_{=:(\star\star)}.
\end{align*}

For the first summand, we get
\begin{align*}
(\star) &= \left\| f^2 - (f_L\circ\xi)^2 \right\|_{L^1_\mu} = \left\|(f-f_L\circ \xi) (f+f_L\circ \xi)\right\|_{L^1_\mu} \\
&\leq \left\|f-f_L\circ \xi\right\|_{L^2_\mu} \left\|f+f_L\circ \xi\right\|_{L^2_\mu}.
\intertext{As $\mu$ is a finite measure on $\X$, we have $L^2_\mu(\X)\subset L^1_\mu(\X)$, hence there exists a $C>0$ such that}
(\star) &\leq C^2 \left\|f-f_L\circ \xi\right\|_{L^1_\mu} \left\|f+f_L\circ \xi\right\|_{L^1_\mu},
\shortintertext{which by Lemma~\ref{lem:f_L1_estimation} can be estimated as}
&\leq C^2 \|\bar{\pi}\|_\infty \varepsilon \left\|f+f_L\circ \xi\right\|_{L^1_\mu}.
\end{align*}
Using the inverse triangle inequality, the remaining factor $\left\|f+f_L\circ \xi\right\|_{L^1_\mu}$ can be estimated as
\begin{align*}
\left\|f+f_L\circ \xi\right\|_{L^1_\mu} &\leq \left\|f - f_L\circ \xi\right\|_{L^1_\mu} + 2\left\|f_L\circ \xi\right\|_{L^1_\mu}\\
&\leq \|\bar{\pi}\|_\infty \varepsilon + 2\left\|f_L\circ \xi\right\|_{L^1_\mu}.
\end{align*}
Overall, we get for the first summand
$$
(\star) \leq 2 C^2 \|f_L\circ \xi\|_{L^1_\mu} \|\bar{\pi}\|_\infty \varepsilon + C^2 \|\bar{\pi}\|_\infty^2 \varepsilon^2
$$

For the second summand, we get
\begin{align*}
(\star\star) &= \left| \left(\mathbb{E}_\mu[f]+\mathbb{E}_\mu[f_L\circ\xi]\right) \left(\mathbb{E}_\mu[f] - \mathbb{E}_\mu[f_L\circ\xi]\right) \right| \\
&\leq \left\|f+f_L\circ\xi\right\|_{L^1_\mu} \left\|f-f_L\circ\xi\right\|_{L^1_\mu}.
\end{align*}
By using Lemma~\ref{lem:f_L1_estimation}, and the same estimation of $\left\|f+f_L\circ\xi\right\|_{L^1_\mu}$ as above, this becomes
\begin{align*}
(\star\star) &\leq 2 \|f_L\circ \xi\|_{L^1_\mu} \|\bar{\pi}\|_\infty \varepsilon + \|\bar{\pi}\|_\infty^2 \varepsilon^2
\end{align*}

Overall, we receive
$$
\left| \var_\mu(f) - \var_\mu(f_L\circ\xi) \right| \leq 2(1+C^2)\|f_L\circ \xi\|_{L^1_\mu} \|\bar{\pi}\|_\infty \varepsilon + (1+C^2) \|\bar{\pi}\|_\infty^2 \varepsilon^2.
$$

\end{proof}

Finally, we show Lemma~\ref{lem:fL_variance}:

\begin{proof}[Proof of Lemma~\ref{lem:fL_variance}]
	We have
	\begin{align*}
		\var_\mu(h\circ \xi) &= \mathbb{E}_\mu[(h\circ\xi)^2]- \left(\mathbb{E}_\mu[h\circ \xi]\right)^2\\
		&= \int_\X \left(h(\xi(x))\right)^2\pi(x)\ts dx - \left(\int_\X h(\xi(x))\pi(x)\ts dx\right)^2.
		\intertext{Using the coarea formula, this becomes}
		&= \int_\Z h^2(z) \int_{\Sigma_\xi(z)} \pi(x) \ts d\sigma_z(x)\ts dz + \left(\int_\Z h(z) \int_{\Sigma_\xi(z)} \pi(x) \ts d\sigma_z(x)\ts dz\right)^2\\
		&= \int_\Z h^2(z) \bar{\pi}(z)\ts dz + \left( \int_\Z h(z) \bar{\pi}(z)\ts dz \right)^2\\
		&= \var_{\bar{\mu}}(h).
	\end{align*}
\end{proof}

\section{Lumpability of Example \ref{sec:example_brownian_motion}}
\label{sec:lumpability_brownian_motion}

For $x,y \in \mathbb{T}$,  we define the shorthand notation
$h^{\sigma}(x,y) := \left( 1 + 2 \sum_{k=1}^\infty \rho^{k^2}
\cos( k (y-x) ) \right)$, where $ \rho = \exp(- \sigma^2/2)$.
Note that since $g^\sigma(x,y) = h^\sigma(x,y) / (2\pi)$ is a density, the function
$h$ is cleary nonnegative.
For all $x,y \in \mathbb{T}^n $, we have
\begin{align}
	\abs{ p^{\tau, \infty}(x,y) - p^{\tau, \sigma}(x,y) }
	&=
	\abs{ \frac{\, h^\tau(x^{(1)},y^{(1)} )}{(2 \pi)^{n} }
		\left(
		1 - \prod_{i=2}^n h^{\tau \sigma}(x^{(i)},y^{(i)})
		\right) }	\nonumber
	\\
	&\leq
	\frac{C(\tau) }{(2\pi)^d}
	\, \max \left\{
	\max_{x,y \in \mathbb{T}^n} \abs{ 1 - \prod_{i=2}^n h^{\tau \sigma}(x^{(i)},y^{(i)}) },
	\min_{x,y \in \mathbb{T}^n} \abs{ 1 - \prod_{i=2}^n h^{\tau \sigma}(x^{(i)},y^{(i)}) }
	\right\} \nonumber
	\\
	&=
	\frac{C(\tau) }{(2\pi)^n}
	\, \max \left\{
	\abs{ 1 -  \left( 1 + 2 \sum_{k=1}^\infty (\pm 1)^k \rho^{k^2} \right)^{\! n-1} }
	\right\} \label{eq:max_wnd}
\end{align}
with the constant $C(\tau) := \max_{x^{(1)},y^{(1)} \in \mathbb{T}} h^\tau(x^{(1)},y^{(1)}) $. Here, we use the fact that for all
$\tau, \sigma > 0$, the
maximum of $\prod_{i=2}^n h^{\tau \sigma}(x^{(i)},y^{(i)})$ is attained at
$x^{(i)} = y^{(i)}$ and its minimum at $\abs{x^{(i)} - y^{(i)}} = \pi$.
When we apply the binomial theorem to
the right hand side of \eqref{eq:max_wnd}, we obtain
\begin{align}
	\abs{ 1 -  \left( 1 + 2 \sum_{k=1}^\infty (\pm 1)^k \rho^{k^2} \right)^{n-1} } &=
	\abs{ \sum_{i = 1}^{n-1}  \binom{n-1}{i} \left(2 \sum_{k=1}^\infty (\pm 1)^k \rho^{k^2} \right)^{\! i} } \nonumber \\
	& = 2(n-1) \left( \sum_{k=1}^\infty  \rho^{k^2} \right) +
		 \mathcal{O} \left(  \left(\sum_{k=1}^\infty  \rho^{k^2} \right)^{\! 2} \right)
		 \label{eq:example_binomial_thm}
	.
\end{align}
We now neglect the higher order terms in \eqref{eq:example_binomial_thm}.
Whenever $\sigma$ is large enough such that $\rho < 1$,
we have by the limit of the geometric series
\begin{align*}
	\left( \sum_{k=1}^\infty  \rho^{k^2} \right) \leq \left( \sum_{k=1}^\infty  \rho^{k} \right) =  \frac{1}{1-\rho} -1
	= \mathcal{ O}(\rho) = \mathcal{ O } \left( \exp(- (\tau \sigma)^2/2)  \right).
\end{align*}
Hence, we have shown
\begin{equation*}
	\max_{x, y \in \mathbb{T}^n} 	\abs{ p^{\tau, \infty}(x,y) - p^{\tau, \sigma}(x,y) } = 
	\left( \exp(- (\tau \sigma)^2/2)  \right).
\end{equation*}
As such, bounding both integrals
in the
norms in the lumpability condition \eqref{eq:lumpability} by the maximum above, we obtain the assertion~\eqref{eq:example_lumpability}.

Note that by neglecting the higher order terms in \eqref{eq:example_binomial_thm},
we essentially ignore the impact of the prefactor depending
on $n$ in terms of the binomial coefficient, constituting the
unknown dependence of the error for flexible $n$ as mentioned in
Section~\ref{sec:slowfast_lumpability}.
\end{appendix}

\end{document}